\documentclass[a4paper,11pt]{article}
\usepackage{amssymb,amsmath,latexsym,amsthm}
\usepackage{color,mathrsfs}
\usepackage{url}
\usepackage{enumerate}

\usepackage[square,sort,comma,numbers]{natbib} 
\usepackage{dsfont}
\usepackage{diagbox}

\usepackage{colortbl}

\usepackage{endnotes}

\usepackage[dvips]{graphics}
\usepackage[xetex]{graphicx}
\usepackage{fancyhdr}
\usepackage{pstricks,pst-plot,pst-node,pstricks-add}

\newtheorem{thm}{Theorem}[section]
\newtheorem{question}[thm]{Question}
\newtheorem{defi}[thm]{Definition}
\newtheorem{corollary}[thm]{Corollary}
\newtheorem{prop}[thm]{Proposition}
\newtheorem{Conjecture}[thm]{Conjecture}
\newtheorem{lemma}[thm]{Lemma}

\theoremstyle{definition}
\newtheorem{remark}[thm]{Remark}

\newtheorem{example}[thm]{Example}

\newcommand{\op}[1]{{\rm{#1}}}

\newcommand{\R}{\mathbb R}

\newcommand{\Z}{\mathbb Z}
\newcommand{\N}{\mathbb N}

\numberwithin{equation}{section}

\def\XXint#1#2#3{{\setbox0=\hbox{$#1{#2#3}{\int}$}
    \vcenter{\hbox{$#2#3$}}\kern-.5\wd0}}

\allowdisplaybreaks
\date{date}
\begin{document}
\title{Dimension reduction techniques for the minimization of theta functions on lattices}
\author{Laurent B\'etermin \quad Mircea Petrache}
\date\today

\maketitle
\begin{abstract}
We consider the minimization of theta functions $\theta_\Lambda(\alpha)=\sum_{p\in\Lambda}e^{-\pi\alpha|p|^2}$ amongst periodic configurations $\Lambda\subset \mathbb R^d$, by reducing the dimension of the problem, following as a motivation the case $d=3$, where minimizers are supposed to be either the BCC or the FCC lattices. A first way to reduce dimension is by considering layered lattices, and minimize either among competitors presenting different sequences of repetitions of the layers, or among competitors presenting different shifts of the layers with respect to each other. The second case presents the problem of minimizing theta functions also on translated lattices, namely minimizing $(\Lambda,u)\mapsto \theta_{\Lambda+u}(\alpha)$,  relevant to the study of two-component Bose-Einstein condensates, Wigner bilayers and of general crystals. Another way to reduce dimension is by considering lattices with a product structure or by successively minimizing over concentric layers. The first direction leads to the question of minimization amongst orthorhombic lattices, whereas the second is relevant for asymptotics questions, which we study in detail in two dimensions.
\end{abstract}

\noindent
\textbf{AMS Classification:}  Primary 74G65; Secondary 82B20 , 11F27  \\
\textbf{Keywords:} Theta functions , Lattices , Layering ,  Ground state.

\section{Introduction and main results}
In the present work we study the problem of minimizing energies defined as theta functions, i.e. Gaussian sums of the form
\begin{equation}\label{eq1}
\theta_\Lambda(\alpha)=\sum_{p\in \Lambda} e^{-\pi\alpha|p|^2},
\end{equation}
among $\Lambda\subset\mathbb R^d$ belonging to the class of lattices (which we will also refer to as ``Bravais lattices'', below) or more generally among some larger class of periodic configurations, constrained to have density $1$. We recall that a lattice is the span over $\mathbb Z$ of a basis of $\mathbb R^d$, and its density is the average number of points of $\Lambda$ per unit volume. This type of problems creates an interesting link between the metric structure of $\mathbb R^d$ and the geometry and arithmetic of the varying lattices $\Lambda$. Our specific focus in this paper is to find criteria based on which the multi-dimensional summation in \eqref{eq1} can be reduced to summation on lower-dimensional sets. We thus select situations in which some geometric insight can be obtained on our minimizations, while at the same time we simplify the problem.
\medskip
\noindent Our basic motivation is the study of the minimization among density one lattices in $\mathbb R^3$, which is relevant for many physical problems (see the recent survey by Blanc-Lewin \cite{Blanc:2015yu}). Note that general completely monotone functions can be represented as superpositions of theta functions with positive weights \cite{Bernstein,BetTheta15}. Therefore our results are relevant for problems regarding the minimization of Epstein zeta functions as well, and for even more general interaction energies. 

\medskip
\noindent  Theta functions are important in higher dimensions, with applications to Mathematical Physics and Cryptography. For applications to Cryptography see for instance \cite{Belfiore} and more generally \cite{ConSloanPacking}. Regarding the applications to Physics, important examples are the study of the Gaussian core system \cite{Stillinger76} and of the Flory-Krigbaum potential as an interacting potential between polymers \cite{FloryKrigbaum}. Recently an interesting decorrelation effect as the dimension goes to infinity has been predicted by Torquato and Stillinger in \cite{TorqStillinger2006,torquatoquantizer}. Furthermore, Cohn and de Courcy-Ireland have recently showed in \cite[Thm. 1.2]{Cohn:2016aa} that, for $\alpha$ enough small and $d\to +\infty$, there is no significant difference between a periodic lattice and a random lattice in terms of minimization of the theta function. Links with string theory have been highlighted in \cite{String}.

\medskip
\noindent  The main references for the minimization problems for lattice energies are the works of Rankin \cite{Rankin}, Cassels \cite{Cassels}, Ennola \cite{Eno2,Ennola}, Diananda \cite{Diananda}, for the Epstein zeta function in $2$ and $3$ dimensions, Montgomery \cite{Mont} for the two-dimensional theta functions (see also the recent developments by the first named author \cite{Betermin:2014fy,BetTheta15}) and Nonnenmacher-Voros \cite{NonnenVoros} for a short proof in the $\alpha=1$ case. See Aftalion-Blanc-Nier \cite{AftBN} or Nier \cite{NierTheta} for the relation with Mathematical Physics, Osgood-Phillips-Sarnak \cite{OPS}, \cite{SarStromb} for the related study of the height of the flat torii, which later entered (see \cite{Betermin:2014rr} for the connection) in the study of the renormalized minimum energy for power-law interactions, via the $\mathcal W$-functional of Sandier-Serfaty \cite{Sandier_Serfaty}, later extended in the periodic case too in work by the second named author and Serfaty \cite[Sec. 3]{JMJ:9726774} to more general dimensions and powers. See also the recent related work \cite{hsss} by Hardin-Saff-Simanek-Su. For the related models of diblock copolymers we also note the important work of Chen-Oshita \cite{CheOshita}. These works except \cite{Mont} mainly focus on energies with a power law behaviour. Regarding the minimization of lattice energies and energies of periodic configurations we also mention the works by Coulangeon, Sch\"urmann and Lazzarini \cite{Coulangeon:2010uq,CoulLazzarini,Coulangeon:kx} who characterize configurations which are minimizing for the asymptotic values of the relevant parameters in terms of the symmetries of concentric spherical layers of the given lattices (so-called spherical designs, for which see also the less recent monographs \cite{BachocVenkov,Venkov1}).

\medskip
\noindent  As mentioned above, there are two main candidates for the minimization of theta functions on $3$-dimensional lattices. They are the so-called body-centered cubic (BCC) and face-centered cubic lattices (FCC), and we describe these two lattices in detail in Section \ref{bccfccsect}. It is known that as $\alpha$ ranges over $(0,+\infty)$, in some regimes the FCC is known to be the minimizing unit density lattice, while in others the BCC is the optimizing lattice, and none of them is the optimum for all $\alpha>0$. See Stillinger \cite{Stillinger79} and the plots \cite[Figures 6 and 8]{Blanc:2015yu} of Blanc and Lewin. Regarding this minimization problem, the critical exponent $\alpha$ is uniquely individuated as $\alpha=1$ by duality considerations (as the dual of a BCC lattice is a FCC one and vice versa, and theta functions of dual lattices are linked by monotone dependence relations). A complete proof of the fact that below exponent $\alpha=1$ the minimizer is the BCC and above it it is the FCC seems to be elusive. A proof was claimed in Orlovskaya \cite{Orlovskaya} but on the one hand most of the heaviest computations are not explicited, while on the other hand providing a compelling geometric understanding of the minimization seems to not be within the goals of that paper (see Sarnak-Str\"ombergsson \cite[Prop. 2]{SarStromb}, and their conjecture \cite[Eq. (43)]{SarStromb}, equivalent to the claimed result \cite{Orlovskaya}).

\medskip
\noindent  Our goal with the present work was first of all to place the BCC and the FCC within geometric families of competitors which span large regions of the $5$-dimensional space of all unit volume $3$-dimensional lattices (see Terras \cite[Sec. 4.4]{Terras}). To do this, we focus on finding possible methods by which theta functions on higher dimensional lattices can be reduced to questions on lower dimensional lattices. 

\medskip
\noindent  A first way to decompose the FCC (or BCC) is into parallel $2$-dimensional lattices, which in this case are either copies of a square lattice $\mathbb Z^2$ or of a triangular lattice generated by $(1,0), (\tfrac12, \tfrac{\sqrt 3}{2})$ (see Section \ref{bccfccstack}). We can perturb such families by moving odd layers with respect to even ones, and try to find methods for checking that the minimum energy configuration is the one giving FCC and BCC. This question reduces to the minimization of the Gaussian sums over $\Lambda + u$ for $\Lambda\subset\mathbb R^d$ a lattice and $u\in\mathbb R^d$ a translation vector (we study this question in high generality in Section \ref{minlu}). Otherwise we can change the period of the repetition of the layers (this and related questions are discussed in a generalized setting in Section \ref{seclayers}).

\medskip
\noindent  Returning to the $3$-dimensional model problem, a second possibility is, while viewing the FCC (or BCC) as a periodic stacking of square lattices, to perturb such lattices by dilations along the axes which preserve the unit-volume constraint. Again the goal is to check that the minimizer is then given by the case of the square lattice. This question leads to the study of mixed formulas regarding products of theta functions (see Section \ref{secorthorombic}). 

\medskip
\noindent  We now pass to discuss in more detail our results.

\subsection{Layer decomposition with symmetry}
\noindent Our first reduction method concerns the study of layered decompositions (see Section \ref{seclayers}).

\medskip
\noindent  Let $\Lambda_0\subset \R^{d-1}$ be a lattice and $H\subset\R^{d-1}$ be a finite set having the same symmetries as $\Lambda_0$ (see Definition \ref{samesymmetries} for a precise statement). Then for each $\mathbf s:\mathbb Z\to H$ and each $t>0$ we can construct a lattice  $\Lambda_{\mathbf s}\subset \R^d$ by stacking copies of $\Lambda_0$ translated by elements of $H$, along the $d$-th coordinate direction at ``vertical'' distance $t$ and by ``horizontally'' translating the $k$-th layer by $\mathbf s(k)$ for all $k\in\mathbb Z$ (see Definition \eqref{defls2}). Then we prove in Section \ref{comparisonef} the following results concerning the minimization of $s\mapsto \theta_{\Lambda_{\mathbf s}}(\alpha)$:
\begin{thm}\label{thm11}
Let $\mathbf s_b:\mathbb Z\to H$ as the $|H|$-periodic map such that $\mathbf s_b|_{\{1,\ldots,|H|\}}= b$ for a bijection $b:\{1,\ldots,|H|\}\to H$. Then the following hold:
\begin{enumerate}
 \item If 
 \begin{equation}\label{ineqalpha}
\alpha\ge \frac{1}{2\pi t^2},
\end{equation}
then the $\mathbf s_b$ as above are minimizers of
$$
\mathbf s\mapsto \theta_{\Lambda_{\mathbf s}}(\alpha)=\sum_{p\in \Lambda_{\mathbf s}} e^{-\pi\alpha |p|^2}
$$
amongst general periodic maps $\mathbf s:\mathbb Z\to H$.
\item If the inequality in \eqref{ineqalpha} is strict then the $\mathbf s_b$ as above exhaust all minimizers among periodic layerings.
\item For all $\alpha>0$ the lattices corresponding to $\mathbf s=\mathbf s_b$ minimize $\mathbf s\mapsto \theta_{\Lambda_{\mathbf s}}(\alpha)$ in the class 
\[
\{\mathbf s: \exists k'\le |H|, \exists b':\mathbb Z/k'\mathbb Z\to H, \mathbf s(i)=b'(i \text{ mod }k')\},
\]
namely amongst all layered configurations of period at most $|H|$.
\end{enumerate}

\medskip
\noindent  
\end{thm}

\begin{figure}[!h]
\centering
\includegraphics[width=8cm]{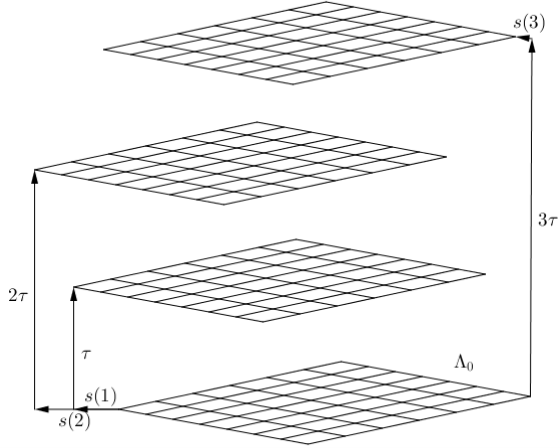}
\caption{Schematic picture of $\Lambda_{\mathbf s}$ as in Theorem \ref{thm11} in dimension $d=3$ (three layers).}
\label{figthm11}
\end{figure}

\noindent  In particular, as an important special case of the second part of the above theorem, we find that $\theta_{FCC}(\alpha)< \theta_{HCP}(\alpha)$ when FCC and HCP have the same density, as claimed in \cite[p. 1, par. 5]{cohnkumarsoft}. This implies that HCP has higher energy than the FCC for all completely monotone interaction functions $f$.

\medskip
\noindent  Our result seems to be the first formalization/proof of this phenomenon in a more general setting. In particular, the above theorem shows that the FCC lattice is a minimizer, for $\alpha$ large enough, in the class $(\Lambda_{\mathbf s})_{\mathbf s}$ where $\Lambda_0$ is an equilateral triangular lattice and that its Gaussian energy, given by the theta function, is lower than the energy of the hexagonal close packed lattice (HCP) for any $\alpha>0$. Our proof partly relies on a (somewhat surprising) reduction to the study of optimal point configurations on riemannian circles by Brauchart-Hardin-Saff \cite{BHS}.

\medskip
\noindent  Theorem \ref{thm11} provides large families amongst which some special lattices can be proved to be optimizers, including an infinite number of non-lattice configurations. In particular we generalize several of the results of \cite{cohnkumarsoft} and of \cite{csbestpack}. We present this link in Section \ref{fiberedpackingsect}, together with further open questions of more algebraic nature, and some natural rigidity questions of lattices under isometries (see also Proposition \ref{proprigidity}), which may be related to isospectrality results as in \cite{csnonuniquetheta}.

\subsection{Minimization of theta functions amongst translates of a lattice}
\noindent Since a way to reduce the dimension that we consider is to look at lattices formed as translated copies of a given lattice, a related interesting question is to study the minimization, for fixed $\alpha>0$, of 
$$
(\Lambda,u)\mapsto \theta_{\Lambda+u}(\alpha)=\sum_{p\in \Lambda} e^{-\pi\alpha|p+u|^2}
$$ 
among lattices $\Lambda\subset \R^d$ and vectors $u\in \R^d$. This is the goal of the Section 3. Translated lattice theta functions appear in several works, as explained in \cite{Regev:2015kq}. They were recently used in the context of Gaussian wiretap channel, and more precisely to quantify the secrecy gain (see \cite[Sec. IV]{Belfiore}). Furthermore, this sum can be viewed as the interaction energy between a point $u$ and a Bravais lattice $\Lambda$. A direct consequence of Poisson summation formula and Montgomery theorem \cite[Thm. 1]{Mont} about the optimality of the triangular lattice for $\Lambda\mapsto \theta_\Lambda(\alpha)$, for any $\alpha>0$, is the following result, proved in Section \ref{UBRSD}:

\begin{prop}\label{prop12}
For any $\alpha>0$, any Bravais lattice $\Lambda\subset \R^d$ of density one and any vector $u\in \R^d$, we have
$$
\theta_{\Lambda+u}(\alpha)\leq \theta_\Lambda(\alpha),
$$
with equality if and only if $u\in \Lambda$.

\medskip
In particular, if $d=2$ and $A_2$ is the triangular lattice of length $1$, then for any Bravais lattice $\Lambda$ such that $|\Lambda|=|A_2|$, any $\alpha>0$ and $u\in \R^2$, it holds
$$
\theta_{A_2+u}(\alpha)\leq \theta_\Lambda(\alpha),
$$
with equality if and only if $u\in A_2$ and $\Lambda=A_2$ up to rotation.
\end{prop}
\noindent This proposition seems to show for the first time that, for any $\alpha>0$ and any given $\Lambda\subset \R^d$, the set of maxima of $u\mapsto \theta_{\Lambda+u}(\alpha)$ is $\Lambda$.

\medskip
\noindent  A particular case of $u$ is the center of the unit cell of the lattice $\Lambda$. The following result shows that this center $c$ is the minimizer of $u\mapsto \theta_{\Lambda+u}(\alpha)$ for any $\alpha>0$, if $\Lambda$ is an orthorhombic lattice, i.e. if the matrix associated to its quadratic form is diagonal. It can be viewed as the generalization of Montgomery result \cite[Lem. 1]{Mont}, which actually proves the minimality for $\beta=1/2$ of $\beta\mapsto \theta_{\Z+\beta}(\alpha)$ for any $\alpha>0$. Furthermore, the following proposition, proved in Sections \ref{secIwasawa} and \ref{Discussdiag}, shows that, for any $\alpha>0$, there is no minimizer of $(\Lambda,u)\mapsto \theta_{\Lambda+c}(\alpha)$.

\begin{prop}\label{prop13}
For any Bravais lattice $\Lambda=M\mathbb Z^d$ with $M\in SL(d)$ decomposed as $M=QDT$, with $Q\in SO(d)$, $D=\text{diag}(c_1,\ldots,c_d),\ c_i> 0,\ \prod_{i=1}^dc_i=1$, $T$ is lower triangular, and for any $u\in\mathbb R^d$ and any $\alpha>0$ there holds
$$
\theta_{\Lambda+u}(\alpha)\geq \theta_{D(\Z +1/2)^d}(\alpha) .
$$
Furthermore, there exists a sequence $A_k=\text{diag}(1,\ldots,1,k,1/k)$, $k\geq 1$, of $d\times d$ matrices such that, for any $\alpha>0$,
$$
\lim_{k\to +\infty}\theta_{A_k(\Z+1/2)^d}(\alpha)=0.
$$
\end{prop}
\noindent This result shows the special role of the orthorhombic lattices. The second part justifies why the minimization of energies of type $\Lambda\mapsto \theta_\Lambda(\alpha)+\delta\theta_{\Lambda+u}(\alpha)$, with $\delta>0$ a real parameter, is interesting: indeed, we see that it is the sum of two theta functions with a competing behaviour with respect to the minimization over Bravais lattices, if $u\not \in \Lambda$. More precisely, for $d=2$, on the one hand $\Lambda\mapsto \theta_\Lambda(\alpha)$ is minimized by the triangular lattice among Bravais lattices of fixed density, and on the other hand $(\Lambda,u)\mapsto \theta_{\Lambda+u}(\alpha)$ does not admit any minimizer on this class of lattices. Therefore, the competition between these two terms will create new minimizers with respect to $\delta$ (see Section \ref{secorthorombic}, and in particular Proposition \ref{cequal1case} in the $\delta=1$ case).

\medskip
\noindent  In dimension $d=1$ and for $\Lambda=\Z$, the classical Jacobi theta functions $\theta_2,\theta_3,\theta_4$ are defined (see \cite[Sec. 4.1]{ConSloanPacking}), for $x>0$, by
\begin{equation}\label{jacobithetadef}
 \theta_2(x)=\sum_{k\in \Z} e^{-\pi (k+1/2)^2 x},\quad \theta_3(x)=\sum_{k\in \Z} e^{-\pi k^2 x}\quad \textnormal{and}\quad \theta_4(x)=\sum_{k\in \Z} (-1)^k e^{-\pi k^2 x}.
 \end{equation}
 Furthermore, we recall the following identity (see \cite[p. 103]{ConSloanPacking}): for any $x>0$,
 \begin{equation}\label{theta24id}
 \sqrt{x}\theta_2(x)=\theta_4\left( \frac{1}{x} \right).
 \end{equation}
 Thus, studying the maximization problem of theta functions among some families of lattices constructed from orthorhombic lattices with translations to the center of their unit cell, we get the following result, proved in Section \ref{proofthm14}, which can be viewed as a generalization of \cite[Thm. 2.2]{Faulhuber:2016aa} in higher dimensions in the spirit of \cite[Thm. 2]{Mont}.

\begin{thm}\label{thm14}
Let $d\geq 1$ and $\alpha>0$. Assume that $\{c_i\}_{1\leq i\leq d}\in (0,+\infty)^d$ are such that $\prod_{i=1}^d c_i =1$ and that not all the $c_i$ are equal to $1$. For real $t$, we define

\begin{minipage}[t]{8cm}
\begin{itemize}
\item $\displaystyle U_4(t)=\prod_{i=1}^d \theta_4(c_i^t \alpha)$,\\
\item $\displaystyle U_2(t)=\theta_{A_t(\Z+1/2)^d}(\alpha)=\prod_{i=1}^d \theta_2(c_i^t \alpha)$,
\item$ \displaystyle Q(t)=\frac{\theta_{A_t(\Z+1/2)^d}(\alpha)}{\theta_{A_t\Z^d}(\alpha)}=\prod_{i=1}^d \frac{\theta_2(c_i^t \alpha)}{\theta_3(c_i^t \alpha)}$,
\end{itemize}
\end{minipage}
\begin{minipage}[t]{8cm}
\begin{itemize}
\item $\displaystyle P_{3,4}(t)=U_3(t)U_4(t)=\prod_{i=1}^d \theta_3(c_i^t \alpha)\theta_4(c_i^t \alpha)$,
\item $\displaystyle P_{2,3}(t)=U_2(t)U_3(t)=\prod_{i=1}^d \theta_2(c_i^t \alpha)\theta_3(c_i^t \alpha)$,
\end{itemize}
\end{minipage}

\noindent where $A_t=diag(c_1^t,...,c_d^t)$ and the classical Jacobi theta functions $\theta_i$, $i\in\{2,3,4\}$ are defined by \eqref{jacobithetadef}. Then for any $f\in \{ U_4, U_2, Q, P_{3,4}, P_{2,3} \}$, we have 
\begin{enumerate}
\item $f'(0)=0$,
\item $f'(t)>0$ for $t<0$,
\item $f'(t)<0$ for $t>0$.
\end{enumerate}
In particular, $t=0$ is the only strict maximum of $f$.
\end{thm}
\noindent The particular case $f=Q$ gives the maximality of the simple cubic lattice among orthorhombic lattices for the periodic Gaussian function (see \cite{Regev:2015kq}) with translation $c$ and fixed parameter.

\subsection{Asymptotic results}
\noindent If $\Lambda$ is not an orthorhombic lattice, then we do not have a product structure for the theta function and the minimization of $u\mapsto \theta_{\Lambda+u}(\alpha)$ is more challenging. It seems that the deep holes of the lattice $\Lambda$ play an important role. For instance, in dimension $d=2$, Baernstein proved (see \cite[Thm. 1]{baernstein}) that the minimizer is the barycentre of the primitive triangle if $\Lambda$ is a triangular lattice. Moreover, numerical investigations, for some $\alpha>0$, show that the minimizer is the center of the unit cell if $\Lambda$ is rhombic (see the definition before Theorem \ref{thm16}). The shapes of naturally occurring crystals may be considered good indicators for that. Also note the numerical study of Ho and Mueller \cite[Fig. 1 and 2]{Mueller:2002aa} detailed in \cite[Fig. 16]{ReviewvorticesBEC}. In the following theorem, proved in Section \ref{secproof15}, we present an asymptotic study, as $\alpha\to +\infty$, of this problem:

\begin{thm}\label{thm15}
Let $\Lambda$ be a Bravais lattice in $\mathbb R^d$ and let $c$ be a deep hole of $\Lambda$, i.e. a solution to the following optimization problem:
\begin{equation}\label{maxneighbordistc}
\max_{c'\in\mathbb R^d}\min_{p\in\Lambda}|c'-p|.
\end{equation}
For any $x\in\mathbb R^d$ there exists $\alpha_x$ such that for any $\alpha>\alpha_x$, 
\begin{equation}\label{minimumcc}
\theta_{\Lambda+c}(\alpha)\leq \theta_{\Lambda+x}(\alpha).
\end{equation}
\end{thm} 
\noindent This result links our study to the one of best packing for lattices and to Theorem \ref{thm11}, as we expect the above minima to be playing the role of $H$ from Theorem \ref{thm11}. The systems corresponding to $\alpha\to +\infty$ are called ``dilute systems" (see \cite{TorquatoGCM08,cohnkumarsoft}) and they correspond to the low density limit of the configuration.

\medskip
\noindent  Furthermore, an analogue of this result is proved, in Section \ref{secproof16}, in dimension $d=2$, as $\alpha\to 0$, by using Poisson summation formula and analysing concentric layers of the lattices. 

\medskip
\noindent Recall first that a two-dimensional lattice is called \emph{rhombic} if up to rotation it is generated by vectors of the form $(a,b),(0,2a)$, and the fundamental rhombus is then the convex polygon of vertices $(0,0),(a,b),(2a,0),(a,-b)$. 

\begin{thm}\label{thm16}
Let $\Lambda$ be a Bravais lattice in $\mathbb{R}^2$. Then the asymptotic minimizers $C$ of $x\mapsto \theta_{\Lambda+x}(\alpha)$ as $\alpha\to0$ are as follows:
\begin{enumerate}
\item If $\Lambda$ is a triangular lattice then $C$ contains only the center of mass of the fundamental triangle.
\item If $\Lambda$ is rhombic and the first layer $C_1$ of the dual lattice $\Lambda^*$ has cardinality $4$ (equivalently, we require that $\Lambda$ is rhombic but not triangular), then $C$ contains only the center of the fundamental rhombus.
\item In the remaining cases consider the second layer $C_2$ of $\Lambda^*$:
\begin{enumerate}
\item If $C_2$ has cardinality $2$ or $6$ then $C$ contains only the center of the fundamental unit cell.
\item Else, there exist coordinates such that if $A$ is the matrix which transforms the unit cell of $\Z^2$ to the unit cell of $\Lambda$, then $C=A\cdot \{(1/2,\tfrac14), (1/2,\tfrac34)\}$.
\end{enumerate}
\end{enumerate}
\end{thm}

\begin{figure}[!h]
\centering
\includegraphics[width=8cm]{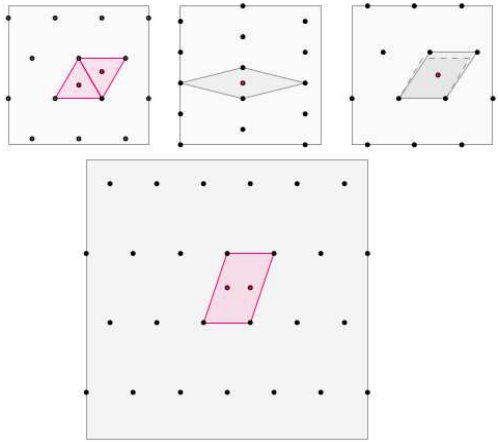}
\caption{Illustration of Theorem \ref{thm16}. First line: examples for cases $1$, $2$ and $3.a$, where in the last case the fundamental cell of a triangular lattice is also shown for comparison. Second line: example in the case $3.b$. The set $C$ corresponds to the red points.}
\label{thm318fig}
\end{figure}

\noindent This theorem is, so far as we know, the only general-framework analogue of the special case considerations done in \cite{cohnkumarsoft}. A similar classification in dimension $3$ seems to be an interesting future direction of research.

\subsection{Orthorhombic-centred perturbations of BCC and FCC and local minimality}

\noindent Recall that the body centred cubic lattices BCC belongs to the class of BCO lattices. Because our first motivation was to study the minimality of BCC and FCC lattices, we prove the following theorem in Section 4 about optimality and non-optimality among body-centred-orthorhombic (BCO, see \cite[Fig. 2.8]{Tilleycrystal}) lattices for $\Lambda\mapsto \theta_\Lambda(\alpha)$. These lattices correspond, e.g., to deformations, by separate dilations along the three coordinate directions, of the BCC.

\begin{figure}[!h]
\centering
\includegraphics[width=8cm]{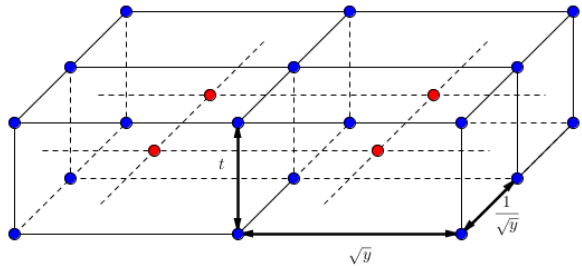}
\caption{Body-Centred-Orthorhombic lattice $L_{y,t}$. The blue and red points correspond to alternating rectangular layers.}
\label{figBCO}
\end{figure}

\noindent The following results is proved in Section \ref{secproof17}:

\begin{thm}\label{thm17}
For any $y\geq 1$ and any $t>0$, let $L_{y,t}$ be the anisotropic dilation. of the BCC lattice (based on the unit cube) along the coordinate axes by $\sqrt{y},1/\sqrt{y}$ and $t$, i.e. 
$$
L_{y,t}:=\bigcup_{k\in \Z} \left( \Z(\sqrt{y},0) \oplus \Z\left(0,\frac{1}{\sqrt{y}}\right) + (\sqrt{y}/2,1/2 \sqrt{y},0)\mathds{1}_{2\Z}(k) +k(0,0,t/2) \right).
$$
We have the following results:
\begin{enumerate}
\item For any $t>0$ and any $\alpha>0$, a minimizer of $y\mapsto \theta_{L_{y,t}}(\alpha)$ belongs to $[1,\sqrt{3}]$.
\item For any $\alpha>0$, there exists $t_0(\alpha)>0$ such that for any $t<t_0(\alpha)$, $y=1$ is not a minimizer of $y\mapsto \theta_{L_{y,t}}(\alpha)$.
\item For any $t<\sqrt{2}$, there exists $\alpha_t$ such that for any $\alpha>\alpha_t$, $y=1$ is not a minimizer of $y\mapsto \theta_{L_{y,t}}(\alpha)$. In particular, for $t=1$, there exists $\alpha_1$ such that
\begin{enumerate}
\item for $\alpha>\alpha_1$, the BCC lattice is not a local minimizer of $\Lambda\mapsto \theta_\Lambda(\alpha)$ among Bravais lattices of unit density,
\item for $\alpha<1/\alpha_1$, the FCC lattice is not a local minimizer of $\Lambda\mapsto \theta_\Lambda(\alpha)$ among Bravais lattices of unit density.
\end{enumerate}
We numerically compute $\alpha_1\approx 2.38$ and we have $\alpha_1^{-1}\approx 0.42$.
\item For $\alpha=1$ and any $t\geq 0.9$, $y=1$ is the only minimizer of $y\mapsto \theta_{L_{y,t}}(\alpha)$. In particular, for $t=1$, the BCC lattice is the only minimizer of $\Lambda\mapsto \theta_\Lambda(1)$ among Bravais lattices $(L_{y,1})_{y\geq 1}$.
\end{enumerate}
\end{thm}

\noindent Applying the point 4 and the fact that the minimizer of $\R^2 \ni u\mapsto \theta_{\Lambda+u}(\alpha)$ is, for any $\alpha>0$, the center of the primitive cell (resp. the center of mass of the primitive triangle) if $\Lambda$ is a square lattice (resp. a triangular lattice), by Proposition \ref{prop13} (resp. \cite[Thm. 1]{baernstein}), we get the following type of result about the local minimality of the BCC and the FCC lattices, here for some values of $\alpha$. As an example result, let 
$$
\mathcal{A}:=\{0.001k; k\in \N, 1\leq k\leq 1000  \} \quad \textnormal{and}\quad \mathcal{A}^{-1}:=\{1/x: x\in\mathcal A\}.
$$
\noindent These sets $\mathcal{A},\mathcal{A}^{-1}$ are just an example. Indeed, our algorithm based on Lemma \ref{lemalgo} allows us to rigorously check the statement of the following theorem, proved in Section \ref{baernstein}, for any chosen $\alpha\leq 1$, for the BCC lattice, and for any $\alpha\geq 1$ for the FCC lattice. We do not have a proof that the statement holds for all such values, but we also do not find any values in these intervals such that our algorithm fails.

\begin{thm}\label{thm18}
Let $\mathcal A, \mathcal A^{-1}$ as above, $\alpha_1>0$ be as in Theorem \ref{thm17} and $\mathcal L_3^o$ be the space of three-dimensional Bravais lattices of density one. Then 
\begin{itemize}
\item For $\alpha\in\mathcal{A}$, BCC is a local minimum of $\Lambda\mapsto \theta_\Lambda(\alpha)$ over $\mathcal L_3^o$. For $\alpha>\alpha_1$ there are two directions in the tangent space of $\mathcal L_3^o$ at the BCC lattice, $T_{BCC} \mathcal L_3^o$, along which BCC is a local maximum and three directions along which it is a local minimum.
\item For $\alpha\in \mathcal{A}^{-1}$, FCC is a local minimum of $\Lambda\mapsto \theta_\Lambda(\alpha)$ over $\mathcal L_3^o$. For $\alpha<1/\alpha_1$ there are two directions in the tangent space of $\mathcal L_3^o$ at the FCC lattice, $T_{FCC} \mathcal L_3^o$, along which FCC is a local maximum and three directions along which it is a local minimum.
\end{itemize}
\end{thm}

\noindent Theorem \ref{thm17} is one of the first complete proofs of the existence of nonlocal regions (here, the spaces of lattices $(L_{y,1})_{y\geq 1}$ and $(L_{y,\sqrt{2}})_{y\geq 1}$ ) on which FCC and BCC are minimal, for small and large values of $\alpha$. Furthermore, Theorem \ref{thm18} supports the Sarnak-Str\"ombergsson conjecture \cite[Eq. (43)]{SarStromb} and Theorem \ref{thm17} gives a first step of the geometric understanding of it. Theorem \ref{thm17} can be also viewed as a generalization of \cite{Faulhuber:2016aa}.

\medskip
\noindent  In particular, we prove that the FCC and the BCC lattices are not local minimizers for any $\alpha>0$, unlike the case of Epstein zeta function (see Ennola \cite{Ennola}). The proof consists of a careful discussion of a sum of two products of theta functions which express the theta function of our competitors. We provide an algorithm for implementing a numerical test to check cases where our theoretical study stops being informative.

\medskip
\noindent  The statement of point 4 of Theorem \ref{thm17} should be compared to \cite[Fig. 1 and 2]{Mueller:2002aa} (see also \cite[Fig. 16]{ReviewvorticesBEC}) in the rectangular lattice case. Indeed, our energy can be rewritten $\theta_{L_{y,t}}(\alpha)=\theta_3(t^2\alpha)\left( \theta_{L_{y}}(\alpha)+\rho_{t,\alpha} \theta_{L_{y}+c_y}(\alpha)\right)$ where $L_{y}$ is the anisotropic dilation of $\Z^2$ along the coordinate axes by $\sqrt{y}$ and $1/\sqrt{y}$, $c_y$ is the center of the unit cell of $L_y$ and $\rho_{t,\alpha}=\theta_2(t^2\alpha)/\theta_3(t^2\alpha)$. Thus, our result shows that, for $\alpha=1$ (which is the Ho-Mueller case \cite{Mueller:2002aa}), $L_1=\Z^2$ is the unique minimizer of $L_y\mapsto \theta_{L_{y,t}}(1)$ if $t$ is large enough (i.e. $\rho_{t,\alpha}$ small enough), whereas for $t$ small enough (i.e. $\rho_{t,\alpha}$ close to $1$), the minimum is a rectangle, as in \cite[Fig. 1 d and e]{Mueller:2002aa}. 

\medskip

\textbf{Structure of the paper.} In Section \ref{seclayers} we study the minimization of $\Lambda\mapsto \theta_\Lambda(\alpha)$ among periodic layering of a given Bravais lattice and we prove Theorem \ref{thm11}. Section \ref{questions25} contains extensions and open questions related to rigidity and fibered packings. In Section \ref{UBRSD} we recall and prove some properties of Regev-Stephens-Davidowitz functions and we prove Proposition \ref{prop12}. In Section \ref{secdegeneracy} we prove a degeneracy property of Gaussian energy. In Section \ref{secIwasawa}, the first part of Proposition \ref{prop13}, i.e. a lower bound of $\theta_{\Lambda+u}(\alpha)$ in terms of rectangular lattice, is proved. In Section \ref{Discussdiag} we prove the second part of Proposition \ref{prop13} about the degeneracy to $0$ of the theta function for a sequence of centred-orthorhombic lattices. In Section \ref{proofthm14} we prove Theorem \ref{thm14}. In Section \ref{secfixedL} we prove some results about the minimization of $u\mapsto \theta_{\Lambda+u}(\alpha)$, $\Lambda$ being triangular or orthorhombic, as well as Theorems \ref{thm15} and \ref{thm16}. In Section \ref{secorthorombic} we prove Theorems \ref{thm17} and \ref{thm18} about the optimality of the FCC and BCC lattices among body-centred-orthorhombic lattices and their local optimality.

\section{Layering of lower dimensional lattices}\label{seclayers}
 \subsection{Preliminaries}
\noindent If $\mathcal C\subset \mathbb R^d$ is a subset, we denote by $\ell \mathcal C$ the dilation by $\ell$ of $\mathcal C$. The \emph{density} (also called \emph{average density}) $d(\mathcal C)$ of a discrete point configuration is defined (See \cite[Defn. 2.1]{Foreman}) as
 \[
 d(\mathcal C):=\lim_{R\to\infty}\frac{|\mathcal C\cap U_R|}{|U_R|},
 \]
where $(U_R)_R$ is a \emph{F\o lner sequence}, namely an increasing sequence of open sets such that for all translation vectors $v\in\R^d$ we have $|U_R\setminus (U_R+v)|/|U_R|\to 0$ as $R\to\infty$. By abuse of notation, we denote by $|\cdot|$ both the cardinality of discrete sets and the Lebesgue measure, depending on the context. Note that the above limit may not exist or may be different for different F\o lner sequences. Such pathological $\mathcal C$ will however not appear in this work, as our $\mathcal C$ all have some type of periodicity. We can thus consider just $U_R=B_R$, the ball of radius $R$ centered at the origin.
 
\medskip
\noindent  We will call here (by an abuse of terminology) a \emph{lattice} in $\mathbb R^d$ any periodic configuration of points. Its \emph{dimension} is the dimension of its convex hull (which is a vector subspace of $\R^d$).
 
\medskip
\noindent  We then call a \emph{Bravais lattice} a subset $\Lambda\subset \mathbb R^d$ such that there exist independent $v_1,\ldots,v_k\in\mathbb R^d$ such that $\Lambda=\op{Span}_{\mathbb Z}\{v_1,\ldots, v_k\}$. 
 
\medskip
\noindent  If $\Lambda\subset \mathbb R^d$ is a Bravais lattice, then its \emph{dual lattice} is defined by
 \[
 \Lambda^*:=\{v\in\op{Span}(\Lambda):\ \forall w\in\Lambda, v\cdot w\in\mathbb Z\}.
 \]
If $\Lambda$ has full dimension $d$ then it can be written as $\Lambda= A\mathbb Z^d$ with $A\in GL(d)$, in which case $\Lambda^*= (A^T)^{-1} \Z^d$. The volume of $\Lambda$, written $|\Lambda|$ is the inverse of its density or, equivalently, the volume of a fundamental domain.
 
 \subsubsection{Energies of general sets and of lattices}
\noindent For countable $\mathcal C\subset\mathbb R^d$ and a real number $\alpha>0$ we may define the possibly infinite series
 \begin{equation}\label{thetageneral}
 \theta_{\mathcal C}(\alpha):=\sum_{p\in\mathcal C} e^{-\pi\alpha|p|^2}.
 \end{equation}
 This coincides with the usual theta function on Bravais lattices. For general functions $f:[0,+\infty)\to [0, +\infty)$ one may also define the analogue interaction energy (generalized in \eqref{defeflambdax} and in the rest of section \ref{comparisonef}) 
 \begin{equation}\label{eflambdageneral}
 E_{f,\mathcal C}(0):= \sum_{p\in\mathcal C} f(|p|^2)\in[0,+\infty].
\end{equation}

\noindent The formulas \eqref{thetageneral} and \eqref{eflambdageneral} will be interpreted as the interaction of a point at the origin with the points from the set $\mathcal C$ corresponding to, respectively, the interaction potentials $f_\alpha(|x|^2) = e^{-\pi\alpha|x|^2}$ and $f(|x|^2)$, respectively. We note that if $\mathcal C=\Lambda$ is instead a Bravais lattice, then $0\in\Lambda$ and the interaction of the origin with $\Lambda$ equals the average self-interaction energy per point.
 
\medskip
\noindent  The special relevance of minimization questions about theta functions \eqref{thetageneral} is due to the well-known result by Bernstein \cite{Bernstein} which allows to treat $E_{f,\mathcal C}(0)$ for any completely monotone $f$ once we know the behavior of all $\theta_{\mathcal C}(\alpha),\alpha>0$. Recall that a $C^\infty$ function $f:[0,+\infty)\to[0,+\infty)$ is called \emph{completely monotone} if 
\[
\text{for all }n\in\mathbb N,\ r\in (0,+\infty)\text{ there holds }(-1)^n f^{(n)}(r)\ge 0 .
\]
Bernstein's theorem states that any completely monotone function can be expressed as
\begin{equation}\label{bernstein}
f(r)=\int_0^{+\infty} e^{-t r}d\mu_f(t),
\end{equation}
where $\mu_f$ is a finite positive Borel measure on $[0,+\infty)$. This representation shows directly that if $\mathcal C_0$  is a minimum of $\mathcal C\mapsto\theta_{\mathcal C}(\alpha)$ for all $\alpha>0$ within a class $\mathscr C$ of subsets of $\mathbb R^d$ then $\mathcal C_0$ is also a minimum on $\mathscr C$ of $E_{f,\mathcal C}(0)$ for all completely monotone $f$ (see \cite{BetTheta15} for some examples in dimension $d=2$).

 \subsubsection{Some notable lattices}
\noindent We denote by $A_2$ the lattice in $\mathbb R^2$ generated by the vectors $(1,0), (1/2, \sqrt 3/2)$. Then the scaled lattice $2^{1/2}3^{-1/4} A_2$ has average density one. The dual lattice $A_2^*$ is isomorphic to $A_2$ and is in fact the $\pi/6$-rotation of $3^{1/2}2^{-1}A_2$.
 \medskip
\noindent The lattices $D_n, n\in\N$, are usually defined as $D_n:=\{(x_1,\ldots, x_n)\in\Z^n:\ \sum x_i=0(\op{mod} 2)\}$. Then $D_n^*$ is formed as $D_n^*=\mathbb Z^n\cup (\mathbb Z^n+(1/2,\ldots,1/2))$. This is again a Bravais lattice, with generators $e_1,\ldots, e_{n-1}, (1/2,\ldots,1/2)$ where $\{e_i:\ 1\le i\le n\}$ is the canonical basis of $\mathbb R^d$. The lattices $2^{-1/n} D_n$ and $ 2^{1/n}D_n^*$ have density one.
 
\medskip
\noindent  Special cases of interest are the following rescaled copy of $D_3$ and $D_3^*$, respectively. In the crystallography community these are the called \emph{face-centered cubic} and \emph{body-centered cubic} lattices. They are formed by adding translated copies of $\mathbb Z^3$: 
\[
 \Lambda_{FCC}:=\bigcup\left\{\mathbb Z^3+\tau:\ \tau\in\{(0,0,0), (1/2,1/2,0),(1/2,0,1/2),(0,1/2,1/2)\}\right\},
\]
\[
\Lambda_{BCC}:=(\mathbb Z^3+(0,0,0))\cup(\mathbb Z^3 + (1/2,1/2,1/2)),
\]
and thus we have $\Lambda_{FCC}=2^{-1}D_3$ and $\Lambda_{BCC}=D_3^*$.

 \subsection{FCC and BCC as pilings of triangular lattices, and the HCP}\label{bccfccsect}
\noindent Define the vectors
 \begin{equation}\label{defabc}
  a=(0,0,0),\quad b=\left(\frac12,\frac{1}{2\sqrt{3}},0\right),\quad c=\left(0,\frac{1}{\sqrt{3}},0\right),\quad \tau=(0,0,t).
 \end{equation}
Then for any bi-infinite sequence ${\mathbf s}:\Z \to\{a,b,c\}$ and $t,\ell>0$, we define the three-dimensional lattice $\Lambda_{{\mathbf s},t,\ell}$ as follows:
\begin{equation}\label{defls}
 \Lambda_{{\mathbf s},t,\ell}=\bigcup_{k\in\mathbb Z}\left(k\tau+\ell {\mathbf s}(k)+\ell A_2\right).
\end{equation}
It is straightforward to check that the volume of a unit cell of $\Lambda_{{\mathbf s},t,\ell}$ is
\[
\rho(\Lambda_{{\mathbf s},t,\ell})=\frac{\sqrt{3}}{2}t\ell^2.
\]
Then define
\begin{equation}\label{s1}
 {\mathbf s}_1(k)=\left\{\begin{array}{lll} a&\text{ if }&k\equiv0 \text{ mod }3\\
                                  b&\text{ if }&k\equiv1 \text{ mod }3\\
                                  c&\text{ if }&k\equiv2 \text{ mod }3\\
               \end{array}\right. .
\end{equation}

\begin{lemma}
The BCC lattices are up to rotation the family of lattices $\Lambda_{{\mathbf s}_1,t,\ell}$ with $\ell/t=2\sqrt{6}$. The FCC lattices are up to rotation the family of lattices $\Lambda_{{\mathbf s}_1,t,\ell}$ with $\ell/t=\frac{\sqrt3}{\sqrt{2}}$.
\end{lemma}
\begin{proof}
Since $\ell/t$ is dilatation-invariant and the two families in the lemma are $1$-dimensional, we restrict to proving that the BCC and FCC lattices belong to the corresponding families in the coordinates up to the rotation which sends the canonical basis to a suitable orthonormal reference frame $(e_1, e_2, e_3)$.\\
To do so in both cases let $e_3=\frac{1}{\sqrt 3}(1,1,1)$. Then choose $\tau_{BCC}=\frac{1}{6}(1,1,1), \tau_{FCC}=\frac{1}{3}(1,1,1)$.

\medskip
\noindent  We first claim that, with the notation $A+B:=\{a+b:\ a\in A, b\in B\}$,
\[
\Lambda_{BCC}\subset \{x+y+z=0\}+\mathbb Z \tau_{BCC},\quad \Lambda_{FCC}\subset \{x+y+z=0\}+\mathbb Z \tau_{FCC}. 
\]
Indeed, $(3,0,0)=(2,-1,-1)+(1,1,1)$ therefore $(1,0,0)\in \{x+y+z=0\}+\mathbb Z \tau_{FCC}$ and by invariance under permutations of coordinates and closure under addition we get $\mathbb Z^3\subset\{x+y+z=0\}+\mathbb Z \tau_{FCC}$. By multiplication by $1/2$ we also get $\frac{1}{2}\mathbb Z^3\subset \{x+y+z=0\}+\mathbb Z \tau_{BCC}$, which directly implies $\Lambda_{FCC}\subset \{x+y+z=0\}+\mathbb Z \tau_{FCC}$. To establish $\Lambda_{BCC}\subset \{x+y+z=0\}+\mathbb Z \tau_{BCC}$ we note that $(1,0,0),(0,1,0)$ together with $\frac{1}{2}(1,1,1)=3 \tau_{BCC}$ generate $\Lambda_{BCC}$ and are all in $\{x+y+z=0\}+\mathbb Z \tau_{BCC}$, and conclude again by closure under addition.

\medskip
\noindent  Next, we note that $\Lambda_{BCC}$ is $3\tau_{BCC}$-periodic and $\Lambda_{FCC}$ is $3\tau_{FCC}$-periodic. We see that $\{x+y+z=0\}\cap\Lambda_{BCC}$ is an intersection of subgroups and thus a lattice, and similarly for $\{x+y+z=0\}\cap\Lambda_{FCC}$. 

\medskip
\noindent  The former contains the equilateral triangle $T_0:=\{(0,0,0),(-1,1,0),(-1,0,1)\}$ and no interior point of it, therefore is a triangular lattice with $\ell=\sqrt 2$. The triangles $T_2:=\{(1,0,0),(0,1,0),(0,0,1)\}$ and $T_4:=\{(1,1,0),(1,0,1),(0,1,1)\}$ are congruent to $T_0$ and contained respectively in 
\[
\left(\{x+y+z=0\}+k\tau_{BCC}\right)\cap \Lambda_{BCC}
\]
for $k=2,4$. By periodicity $T_1=T_4-3\tau_{BCC}$ is contained in the similar slice with $k=1$. We see that \[(1,0,0)=\frac13(2,-1,-1) + 2\tau_{BCC}\text{ and }(1,1,0)=\frac13(1,1,-2) + 4\tau_{BCC}\]
and that $\frac13(2,-1,-1)$ is a $\pi/3$-rotation of $\frac13(1,1,-2)$ clockwise about $\tau_{BCC}$, therefore we may let $e_2$ parallel to $(1,1,-2)$ and $e_1$ such that $(e_1,e_2,e_3)$ is an positive rotation of the canonical basis, and we find that up to such rotation $\Lambda_{BCC}=\Lambda_{{\mathbf s}_1,\frac{1}{2\sqrt 3},\sqrt 2}$. 

\medskip
\noindent  For FCC analogously to above, we check that the slices $\left(\{x+y+z=0\} + k\tau_{FCC}\right)\cap\Lambda_{FCC}$ for $k=0,1,2$ contain respectively the triangles $T_0':=\frac12\{(0,0,0),(0,-1,-1),(-1,0,-1)\}$, $T_1':=\frac12\{(1,1,0),(1,0,1),(0,1,1)\}$ and $T_2':=(1,1,1)-T_1'$; then by similar computations we find a positive rotation that brings $\Lambda_{FCC}$ to $\Lambda_{{\mathbf s}_1,\frac{1}{\sqrt3},\frac{1}{\sqrt2}}$.
\end{proof}
\noindent If to fix scales we require $|FCC|=|BCC|=1$ then we obtain 
\[\left\{\begin{array}{ll}
t_{BCC}=2^{-2/3} 3^{-1/2},\ &\ell_{BCC} = 2^{5/6}\\
t_{FCC}=2^{2/3} 3^{-1/2},\ &\ell_{FCC} = 2^{1/6}.
\end{array}
\right.
\]
We also recall that, by definition, the hexagonal closed packing (HCP) configuration is the non-lattice configuration obtained from the FCC by using a different shift sequence. It is defined as $\Lambda_{{\mathbf s}_2,t,\ell}$, for $\ell/t=\frac{\sqrt3}{\sqrt{2}}$ and
\begin{equation}\label{s2}
 {\mathbf s}_2(k)=\left\{\begin{array}{lll} a&\text{ if }&k\equiv0 \text{ mod }2\\
                                  b&\text{ if }&k\equiv1 \text{ mod }2.\\
               \end{array}\right. .
\end{equation}
\subsection{FCC and BCC as layerings of square lattices}\label{bccfccstack}
\noindent We note here that we may also consider the FCC and BCC to be simpler layerings of square lattices $\Z^2$. Let 
\[
a=(0,0,0), b=(1/2,1/2,0), \tau=(0,0,t),\quad \mathbf s:\Z\to\{a,b\}.
\]
Then define here
\[
\Lambda_{\mathbf s,t}:=\bigcup_{k\in\Z}(k\tau + \mathbf s(k)+\Z^2).
\]
We then easily find that with ${\mathbf s}_2$ defined as in \eqref{s2} the following holds:
\begin{lemma}
The BCC lattices are the family of lattices $\Lambda_{{\mathbf s}_2,t}$ with  $t=1/2$ and up to rotation the FCC lattices are the family $\Lambda_{{\mathbf s}_2,t}$ with $t=\sqrt 2$.
\end{lemma}

\subsection{Comparison of general $E^f$ for periodically piled configurations - Proof of Theorem \ref{thm11}}\label{comparisonef}

\noindent Let $f:[0,\infty)\to[0,\infty)$ be a fixed function. We define for a lattice $\Lambda_0\subset\mathbb R^{d-1}$ and $x\in \R^{d-1}$
\begin{equation}\label{defeflambdax}
E_{f,\Lambda_0}(x):=\sum_{p\in\Lambda_0}f(|p+x|^2)\in[0,\infty].
\end{equation}
In the particular case $f(r)=f_\alpha(r)=e^{-\pi\alpha r}$, we define
\begin{equation}\label{thetatranslated}
\theta_{\Lambda_0+x}(\alpha):=E_{f_\alpha,\Lambda_0}(x)=\sum_{p\in \Lambda_0} e^{-\pi \alpha |p+x|^2}.
\end{equation}
Consider now a periodic function $\mathbf s:\mathbb Z\to\mathbb R^{d-1}\subset \mathbb R^d$ and a vector $\tau\in(\mathbb R^{d-1})^\perp$. Define ${\mathbf s}_\tau(k)={\mathbf s}(k)+k\tau$ and the configuration 
\begin{equation}\label{defls2}
 \Lambda_{\mathbf s}=\bigcup_{k\in\mathbb Z}\left({\mathbf s}_\tau(k)+\Lambda_0\right).
\end{equation}
We define the following average $f$-energy per point, where $P$ is any period of ${\mathbf s}$:
\begin{equation}\label{defefl}
 E^f(\Lambda_{\mathbf s}):=\frac{1}{P}\sum_{h=1}^P\sum_{k\in\mathbb Z}E_{f,\Lambda}({\mathbf s}_\tau(h) - \mathbf s_\tau(k)).
\end{equation}
It is easy to verify that the value of above sum does not depend on the period $P$ that we chose: if $P',P''$ are distinct periods of $\mathbf s$ and $\op{mcf}(P',P'')=P$ then $P$ is also a period, and we may use the fact that ${\mathbf s}_\tau(h+aP)-{\mathbf s}_\tau(k+aP) = {\mathbf s}_\tau(h)-{\mathbf s}_\tau(k)$ for $a=1,\ldots,P'/P$ to rewrite the sums in \eqref{defefl} for $P'$ and obtain that they equal those for $P$.

\medskip
\noindent  We now consider a finite set $H\subset\mathbb R^{d-1}$ and assume that \emph{the lattice $\Lambda_0$ has the same symmetries as $H$} in the sense of the following definition:
\begin{defi}\label{samesymmetries}
Let $d\geq 2$, $H\subset \R^{d-1}$ and $\Lambda_0\subset \R^{d-1}$ be a Bravais lattice. We say that $\Lambda_0$ has the same symmetries as $H$ if, for any $x,y,w,z\in H$ with $x\neq y$ and $w\neq z$, there exists a bijection $\phi:\Lambda_0\to\Lambda_0$ such that, for all $p\in\Lambda_0$
\begin{equation}\label{eqsamesymmetries}
 |p+x-y|=|\phi(p)+w-z|.
\end{equation}
\end{defi}
\begin{example}
 Consider the triangular Bravais lattice $A_2$. Then the FCC lattice is $\Lambda_{{\mathbf s}_1}$ and the HCP lattice is $\Lambda_{{\mathbf s}_2}$ for ${\mathbf s}_1, {\mathbf s}_2$ as in \eqref{s1} and \eqref{s2}. Then the maps ${\mathbf s}_1, {\mathbf s}_2$ take values in sets composed of the vectors $a,b$ or $a,b,c$ respectively, which form possible choices of $H$ as above.
\end{example}
\noindent For more discussions regarding the above Definition \ref{samesymmetries}, see Section \ref{rigidityssec}.\\

\medskip
\noindent  We now prove Theorem \ref{thm11} concerning the optimality of a special type of function ${\mathbf s}$.

\begin{proof}[Proof of Theorem \ref{thm11}] We prove now the first part of the theorem. Let $\alpha>0$ and $t>0$ such that $\alpha\geq \frac{1}{2\pi t^2}$ and ${\mathbf s}:\Z\to H$ be a periodic map, where $H$ has the same symmetries as $\Lambda_0$. Let us prove that
$$
\theta_{\Lambda_{\mathbf s}}(\alpha)\geq \theta_{\Lambda_{\mathbf s_b}}(\alpha)
$$
for any ${\mathbf s}_b$ that is $|H|$-periodic and such that ${\mathbf s}_b|_{\{1,\ldots,|H|\}}=\mathbf b$ is a bijection into $H$.

\medskip
\noindent  \textbf{Step 1.1: }\textit{Reduction to a $1$-dimensional problem}\\
Let $f:[0,+\infty)\to [0,+\infty)$ be such that for some $\epsilon, r_0>0$ we have $f(r)=O(r^{-d/2-\varepsilon})$ for $r\geq r_0$ and $f(0) \in \R$. This condition will be easily verified in our case, and is required in order for the sums below to be unconditionally convergent. Let $P$ be a period of $\mathbf s$ which is also a multiple of $|H|$. Keeping in mind \eqref{defls2} we find that:
 \[
  E^f(\Lambda_{\mathbf s}) = \frac{1}{P}\sum_{h=1}^{P}\sum_{k\in\mathbb Z}\sum_{p\in\Lambda_0}f\left((h-k)^2t^2 + |{\mathbf s}(h) -{\mathbf s}(k) + p|^2\right) .
 \]
For a constant sequence ${\mathbf s}(k)\equiv a\in H$ we write $\Lambda_a$ the corresponding lattice and we have
\[
 E^f(\Lambda_{\mathbf s\equiv a})=\sum_{k\in\mathbb Z}\sum_{p\in\Lambda_0}f\left(\left|k\tau+p\right|^2\right).
\]
Therefore, we obtain
\begin{eqnarray*}
E^f(\Lambda_{\mathbf s\equiv a})-E^f(\Lambda_{\mathbf s})&=&\frac{1}{P}\sum_{h=1}^P\sum_{k\in\mathbb Z}\sum_{p\in\Lambda_0}\left[ f\left((h-k)^2t^2 +|p|^2\right)-f\left((h-k)^2|t|^2 + |{\mathbf s}(h) -{\mathbf s}(k) + p|^2\right)\right].
\end{eqnarray*}
We will use the hypothesis that $H$ has the same symmetries as $\Lambda_0$: denote $l=a-b$ for any $a\neq b\in H$ and note that the following expression does not depend on the choice of such $a,b$ due to \eqref{eqsamesymmetries}:
\begin{equation}\label{tildef}
 \tilde F(h):=\sum_{p\in\Lambda_0}\left[f\left(h^2t^2 + |p|^2\right) - f\left(h^2t^2+|l+p|^2\right)\right].
\end{equation}
With this notation we obtain, indicating $\Delta(v,w)=0$ if $v=w$ and $\Delta(v,w)=1$ otherwise, (using also the fact that we may omit terms with $\ell\in P\mathbb Z$ below because they are zero by the choice of $\Delta$ and by periodicity)
\begin{eqnarray*}
E^f(\Lambda_{\mathbf s\equiv a})-E^f(\Lambda_{\mathbf s})&=&\frac{1}{P}\sum_{h=1}^{P}\sum_{k\in\mathbb Z}\Delta(\mathbf s(h), \mathbf s(k))\tilde F(k-h)\\
 &=&\frac{1}{P}\sum_{h=1}^{P}\sum_{\ell\in\mathbb Z}\Delta(\mathbf s(h), \mathbf s(\ell+h))\tilde F(\ell)
 \\
 &=&\frac{1}{P}\sum_{h=1}^{P}\sum_{\ell=1}^\infty\left(\Delta(\mathbf s(h),\mathbf s(h+\ell)) +\Delta(\mathbf s(h), \mathbf s(h-\ell))\right)\tilde F(\ell).
\end{eqnarray*}
Now define 
\begin{equation}\label{defF}
 F(h'):=\sum_{n=0}^\infty \tilde F(nP+h')\quad\text{ for }\quad h'=1,\ldots,P.
\end{equation}
We may define finally, with the change of variable $\ell=nP+h', n\in\mathbb N, h'=1,\ldots,P$ and using the fact that $P$ is a period of $\mathbf s$:
\begin{eqnarray}
E^f(\Lambda_{\mathbf s})-E^f(\Lambda_{\mathbf s\equiv a})&=&\frac{1}{P}\sum_{h=1}^{P}\sum_{h'=1}^{P-1}\left(\Delta(\mathbf s(h),\mathbf s(h+h')) +\Delta(\mathbf s(h), \mathbf s(h-h'))\right)F(h').\label{formulabare}
\end{eqnarray}
\medskip
\noindent  \textbf{Step 1.2: }\textit{Computations and convexity in the case of theta functions.}\\
Note that the formula defining $F(h')$ works also for arbitrary $h'\in[0,P)$. We now check that \textit{for $f(|x|^2) = e^{-\pi\alpha|x|^2}$ and large enough $\alpha$ the function $F$ is decreasing convex on $[1,P-1]$.} Indeed in this case we get 
\begin{eqnarray*}
F(x)&=&\left(\sum_{n=0}^\infty e^{-\pi\alpha t^2(nP + x)^2}\right)(\theta_{\Lambda_0}(\alpha)-\theta_{\Lambda_0+l}(\alpha)),\\
F'(x)&=&-2\pi\alpha t^2\left(\sum_{n=0}^\infty (nP + x)e^{-\pi\alpha t^2(nP + x)^2}\right)(\theta_{\Lambda_0}(\alpha)-\theta_{\Lambda_0+l}(\alpha)),\\
F''(x)&=&\left(\sum_{n=0}^\infty (4\pi^2\alpha^2 t^4(nP + x)^2-2\pi\alpha  t^2)e^{-\pi\alpha t^2(nP + x)^2}\right)(\theta_{\Lambda_0}(\alpha)-\theta_{\Lambda_0+l}(\alpha)),\\
\end{eqnarray*}
and we note that by Proposition \ref{translationregevstephens} below, $\theta_{\Lambda_0+l}(\alpha)<\theta_{\Lambda_0}(\alpha)$ for all $\alpha>0, l\notin \Lambda_0$ and so $F$ is decreasing in $x$ for $x\ge1$. It is also convex on this range because
$$
\alpha\geq \frac{1}{2\pi t^2} \Longrightarrow \forall n\geq 0, \forall x\geq 1, 4\pi^2\alpha^2 t^4(nP + x)^2-2\pi\alpha  t^2\geq 0 \Longrightarrow \forall x\geq 1,F''(x)\geq 0
$$
\noindent \textit{We may also assume that $F$ is convex on the whole range $(0,P]$ up to modifying it on $(0,P]\setminus [1,P-1]$.} Indeed the estimates that we obtain below are needed only to compare energies of integer-distance sets of points, thus they concern $F$'s values restricted only to the unmodified part.

\medskip
\noindent  \textbf{Step 1.3:} \textit{Relation to the minimization on the circle and conclusion of the proof.}\\
We know that $F$ depends on $P$ only, so in order to compare finitely many different sequences $\mathbf s$ we may take $P$ to be one of their common periods, and the minimizer of \eqref{formulabare} will coincide with the minimum of $E^f$ among lattices for which $P$ is a period. We relate \eqref{formulabare} to the energy minimization for convex interaction functions on a curve, studied in \cite{BHS}.

\medskip
\noindent  We will consider a circle $\Gamma$ of length $P$, which we also assume to be a multiple of $2d$, and on which the equidistant distance $1$ points are labelled by $H$, in the order defined by $\mathbf s(k)$ with $1\le k\le P$. For $p\in H$ we define $A_p$ to be the sets of points with label $p$. Let $\pi_{\mathbf s}:=\{A_p: p\in H\}$. Moreover $d(x,y)$ denotes the arclength distance along the circle $\Gamma$. Then \eqref{formulabare} gives the same value as 
\begin{equation}\label{barenew}
 E^F(\pi_{\mathbf s}):=\frac{1}{P}\sum_{p\in H}\sum_{x\neq y\in A_p}F(d(x,y)) .
\end{equation}
Therefore the minimum of \eqref{formulabare} corresponds to the minimum of $E^F(\pi_{\mathbf s})$ for $\mathbf s$ as above. 
We compare each of the sums over $A_p$ in \eqref{barenew} with the minimum $F$-energy of a set of points on $\Gamma$ of cardinality $|A_p|$. The latter minimum is realized by points at equal distances along $\Gamma$, by \cite[Prop. 1.1(A)]{BHS}. Denoting the minimum $F$-energy of $N$ points on $\Gamma$ by 
\[
 \mathcal E^F(N):=\min\left\{\sum_{x\neq y\in \Omega}F(d(x,y)):\ \Omega\subset\Gamma,\, |\Omega|=N\right\},
\]
we have then
\[
 P(E^F(\pi_{\mathbf s})-F(0))\ge \sum_{p\in H}\mathcal E^F(|A_p|).
\]
Note that up to doubling $P$ we may assume that the numbers $|A_p|$ with $p\in H$ are all even and add up to $P$, since they correspond to instances of $\mathbf s(k)=p$ along a period $P$. Recall that we already assumed that $P$ is a multiple of $|H|$ before. We then claim that 
\begin{equation}\label{mineabc}
 |H|\ \mathcal E^F\left(\frac{P}{|H|}\right)=|H|\ \mathcal E^F\left(\frac{1}{|H|}\sum_{p\in H}|A_p|\right)\le \sum_{p\in H}\mathcal E^F(|A_p|).
\end{equation}
To see this first observe (cf. \cite[Eq. (2.1)]{BHS}) that for even $N$
\[
 \mathcal E^F(N)= N\sum_{n=1}^{N/2}F\left(\frac{P}{N}n\right)- NF(P).
\]
Therefore 
\begin{eqnarray*}
 |H|\ \mathcal E^F(P/|H|) &\le& \sum_{p\in H}\mathcal E^F(|A_p|)\\
 &\Leftrightarrow&\\
 |H|\sum_{n=1}^{P/|H|}F(|H|\ n)&\le&\sum_X\frac{|H|\ X}{P}\sum_{n=1}^{X/2}F\left(\frac{P}{X}\ n\right), 
\end{eqnarray*}
where in the last sum $X$ is summed over $\{|A_p|, p\in H\}$. Now recalling the definition \eqref{defF} of $F$ and \eqref{tildef} of $\tilde F$, we see that the above last line can be rewritten in the following form
\begin{equation}\label{mintfnt}
\sum_{n\in\mathbb N}\tilde F(|H|n)\le\sum_{p\in H}\sum_{n\in\mathbb N}t_p\tilde F\left(\frac{|H|n}{t_p}\right),
\end{equation}
where the $t_p$ equal $|H|\ |A_p|/P$ and satisfy $\sum t_p=1$. Note that in Step 1.2 in particular we proved that $\tilde F$ is convex, which in turn implies that for all $c>0$ the function $t\mapsto t\tilde F(c/t)$ is also convex and by additivity so is $t\mapsto\sum t\tilde F(|H|\ n/t)$. Therefore \eqref{mintfnt} is true and \eqref{mineabc} holds. In particular \eqref{barenew} is minimized for the partition $\pi_{\mathbf s}$ where each of the $A_p$ consists of $P/|H|$ equally spaced points each, which corresponds to the choice $\mathbf s=\mathbf s_b$ for some bijection $b:\mathbb Z/|H|\mathbb Z\to H$. Therefore any such $\mathbf s_b$ minimizes $\mathbf s\mapsto E^f(\Lambda_{\mathbf s})$ amongst $P$-periodic $\mathbf s$ where $P$ is a period of $\mathbf s$ multiple of $|H|$. Since any periodic $\mathbf s$ has one such period and $E^f(\Lambda_{\mathbf s})$ is independent on the chosen period, the thesis follows. 

\medskip
\noindent  \textbf{Final remark.} If for some choice of $f$ the $\tilde F, F$ of Step 1.1 are decreasing and strictly convex, then in this case the cited result of \cite{BHS} as well as the inequalities \eqref{mintfnt} and \eqref{mineabc} become strict outside the configurations corresponding to $\mathbf s={\mathbf s}_1$, implying that ${\mathbf s}_1$ is the unique minimizer for any fixed $P$. This gives the uniqueness part of the theorem's statement.

\medskip

\noindent We now prove the last part of the theorem which shows the optimality of $\mathbf s_b$ in a smaller class, but for all $\alpha>0$.

\medskip
\noindent  \textbf{Step 2.1.}
To a layer translation map $\mathbf s: \mathbb Z\to H$ and to a layer index $\bar k$ we associate the function $\Delta_{\mathbf s, \bar k}:\mathbb Z\to\{0,1\}$ defined as follows:
\[
\Delta_{\mathbf s,\bar k}(k):=\left\{\begin{array}{ll} 0&\text{ if }\mathbf s(k)=\mathbf s(\bar k),\\ 1&\text{ if }\mathbf s(k)\neq \mathbf s(\bar k).\end{array}\right. .
\]
Then define as usual $f(|x|^2)=e^{-\pi\alpha|x|^2}$ and note that like in the proof of the first part of the theorem the difference $E^f(\Lambda_0) - E^f(\Lambda_{\mathbf s})$ is the average for $\bar k\in\{0,\ldots,P-1\}$ of the sums on $k\in\mathbb Z, p\in\Lambda_0$
\begin{equation}\label{defdeltabarf}
\left\{\begin{array}{ll}
f((\bar k-k)^2 |t|^2 + |p|^2) - f((\bar k-k)^2 |t|^2 + |p+u|^2)&\text{ if } \Delta_{\mathbf s,\bar k}(k)=1,\\
0 &\text{ if }\Delta_{\mathbf s,\bar k}=0.
       \end{array}
\right.
\end{equation}
We note again that by Proposition \ref{translationregevstephens} below, there holds
\begin{equation}\label{comparisonlayerslice}
\sum_{p\in \Lambda_0} (e^{-\pi\alpha(h^2|t|^2 + |p|^2)}-e^{-\pi\alpha(h^2|t|^2 + |p+u|^2)}) =e^{-\pi\alpha h^2|t|^2}(\theta_{\Lambda_0}(\alpha) -\theta_{\Lambda_0+u}(\alpha)) > 0.
\end{equation}
The rest of the proof is divided into two steps. 

\medskip
\noindent  \textbf{Step 2.2.} Assume that $k'=|H|$. We then prove that the choice of $\mathbf s$ for which $b'=b$ is injective is minimizing with respect to all other $k'$-periodic choices of $b'$. Indeed, in this case we see that for each $\bar k$ the function $\Delta_{b, \bar k}$ takes values $0,1,\ldots,1$ within one period, whereas if $b'$ is not injective then the function $\Delta_{b',\bar k}$ takes for at least one value of $\bar k$ values $0$ for both $k=\bar k$ and for another choice $k<|H|$. It thus suffices that the $|H|$-tuple of values $0,1,\ldots,1$ is the one for which the contributions \eqref{defdeltabarf} corresponding to one single period take the smallest possible value. This is a direct consequence of Step 2.1.

\medskip
\noindent \textbf{Step 2.3.} Assume that $b, b'$ are both bijections and $k'<|H|$. We then show that $\theta_{\Lambda_{\mathbf s_{b'}}}(\alpha)>\theta_{\Lambda_{\mathbf s_b}}(\alpha)$ for all $\alpha>0$. Indeed in that case we have that $\Delta_{b,\bar k}(\bar k +k), \Delta_{b',\bar k}(\bar k +k)$ both are independent of $\bar k$. Thus we compare them over a period of $k'|H|$ starting at $k=0$ and with $\bar k=0$. We see that since $k'<|H|$ then to each $k\in\{1,\ldots,k'|H|-1\}$ where $\Delta_{b,0}(k)=0$ and $\Delta_{b',0}(k)=1$ we can injectively associate a value $k''\in\{1,\ldots,k-1\}$ at which $\Delta_{b,0}(k'')=1, \Delta_{b',0}(k'')=0$. As the contributions to $\theta_{\Lambda_{\mathbf s_{b'}}}(\alpha)-\theta_{\Lambda_{\mathbf s_b}}(\alpha)$ of the form \eqref{comparisonlayerslice} corresponding to these two values $k,k''$ are
\[
(e^{-\pi\alpha (k'')^2|t|^2}-e^{-\pi\alpha k^2|t|^2})(\theta_{\Lambda_0}(\alpha) -\theta_{\Lambda_0+u}(\alpha))>0,
\]
because the first term is positive because $r\mapsto e^{-\pi\alpha r^2}$ is decreasing, whereas the second term is positive by Proposition \ref{translationregevstephens}.
\end{proof}

\begin{remark}
The same optimality of the lattices $\Lambda_{\mathbf s_b}$ among all periodically layered lattices holds also for more general energies $E^f$, as soon as we can ensure that the function $F$ present in \eqref{defF} via $\tilde F$ from \eqref{tildef} is strictly decreasing and convex on $[1,P-1]$. Indeed in this case the Step 1.2 of the proof can be replaced and the rest of the proof goes though verbatim.
\end{remark}

\begin{example}[Comparison between FCC and HCP]\label{fcchcp}
As a consequence of the above, for any $\alpha>0, t,\ell$, $\theta_{\Lambda_{{\mathbf s}_1, t,\ell}}(\alpha)\leq \theta_{\Lambda_{{\mathbf s}_2, t,\ell}}(\alpha)$. In particular, $\theta_{FCC}(\alpha)< \theta_{HCP}(\alpha)$ when $|FCC|=|HCP|$. This implies that HCP has higher energy than the FCC for all completely monotone interaction functions $f$.

\end{example}

\subsection{Questions, links and possible extensions}\label{questions25}
\subsubsection{A rigidity question related to Definition \ref{samesymmetries}}\label{rigidityssec}
\noindent Note that a case when $H$ has the same symmetries as $\Lambda_0$ is when for each $x\neq y, w\neq z\in H$ there exists an affine isometry of $\Lambda_0$ sending $x$ to $w$ and $y$ to $z$. The following rigidity question, concerning the distinction between affine isometries of $\mathbb R^{d-1}$ and isometries of a given lattice of dimension $d-1$, are open as far as we know (recall that distinct lattices may be isometric, as in \cite{csnonuniquetheta}).
{\textbf\textit
\begin{question}
If $\Lambda_0$ is a homogeneous discrete set and $H$ is as in Definition \ref{samesymmetries}, then does it follow that we can find bijections $\phi$ in \eqref{samesymmetries} which are actually restrictions of isometries of $\mathbb R^{d-1}$ to $\Lambda_0$?
\end{question}
Recall that a set $A\subset \mathbb R^m$ is called \emph{homogeneous} if it is the orbit of a point by a subgroup of the isometries of $\mathbb R^m$. 
}
Note the following related positive result, whose proof is however nontrivial:

\begin{prop}\label{proprigidity}
For any lattice $\Lambda\subset \mathbb R^d$ a length-preserving bijection $\phi:\Lambda\to\Lambda$ is the restriction of an affine isometry of $\R^d$.
\end{prop}
\begin{proof}
 In fact our proof will show that any length-decreasing bijection of a lattice is the restriction of an affine isometry. We consider first the set $V_1$ of shortest vectors of $\Lambda$ and we find that for all $e\in\Lambda$, the restriction $\phi|_{e+V_1}$ must be a bijection from this set to $\phi(e)+V_1$, thus it is uniquely determined by a permutation, which a priori could depend on $e$. 
 
 \medskip
\noindent  \textbf{Step 1.} We will show that this permutation does not depend on $e$. To start with, we note that for each $e\in \Lambda$ and $v\in V_1$ the restriction of $\phi$ to $(e + \mathbb R v)\cap \Lambda$ is an isometry. Indeed we have 
 \[
  |\phi(kv+e)-\phi(e)| \le \sum_{k'=0}^{k-1} |\phi((k'+1)v+e) - \phi(k' v+e)|= k |v|,
 \]
 with equality if and only if all vectors $\phi((k'+1)v + e) - \phi(k'v+e)$ are positive multiples of each other. As $\phi:e + V_1\to \phi(e)+V_1$ is a permutation, the above vectors all belong to $V_1$, which does not contain vectors of different length. This implies that all these vectors are equal, and thus $\phi$ restricts to an isometry on $e + \mathbb Z v$ as desired. 
 
 \medskip
\noindent  \textbf{Step 2.} Next, for independent $v_1\neq v_2\in V_1$ we claim that $\phi$ restricts to an isometry on $e + \mathbb Z v_1+ \mathbb Z v_2$. We start by noticing that due to the previous step, $\phi|_{e+\mathbb Z v_1}$ is an isometry and thus it has the form $\phi(e+kv_1)=\phi(e) + kw_1$ and similarly $\phi(e+ hv_1+kv_2)=\phi(e)+ hw_1 + k w_2(h)$ with $w_1\in V_1$ is a fixed vector and $w_2:\mathbb Z\to V_1$ is a function which we desire to prove is constant. By a slight abuse of notation we define $w_2(0):=w_2$ and we desire to prove that for all $h\in\mathbb Z$ there holds $w_2(h)=w_2$. Indeed assume that $h$ is a value such that $w_2(h)\neq w_2$. Then we must still have for all $h\in\mathbb Z$, using the $1$-Lipschitz property of $\phi$,
 \[
 |hw_1 + k(w_2(h)-w_2)|=|\phi(e+ h v_1 + k v_2) - \phi(e + kv_2)| \le h|v_2|.
 \]
As $k\to\infty$ we find that this proves $w_2(h)=w_2$, which proves our claim.

\medskip
\noindent  \textbf{Step 3.} Similarly to the previous step we find that $\phi|_{(e+\op{Span}V_1)\cap\Lambda}$ coincides with an affine isometry for all $e\in\Lambda$. Assuming that $\op{Span}V_1\cap\Lambda\neq \Lambda$, let $V_2$ be the set of shortest vectors of $\Lambda\setminus \op{Span}V_1$. Then at each $e\in\Lambda$ the function $V_1\cup V_2\ni v\mapsto \phi(e+v)-\phi(v)$ is a permutation of $V_1\cap V_2$. We saw that $V_1$ is sent to itself and thus this function induces a permutation of $V_2$. Steps 1 repeats for vectors in $V_2$ to show that $\phi$ restricts to an isometry on each $e+\mathbb Z v, v\in V_2$, and then the reasoning of Step 1 allows to show that $\phi$ restricts to an affine isometry on each $(e +\op{Span}V_1 +\op{Span}V_2)\cap \Lambda$ for each $e\in\Lambda$. 

\medskip
\noindent  \textbf{Step 4.} For each $k\ge1$ we can then repeat Step 3, and apply it to the sets $V_{k+1}$ of shortest vectors of $\Lambda\setminus\op{Span}(V_k)$ defined iteratively after $V_1$. This allows to prove, after finitely many steps, that $\phi$ restricts to an affine isometry on the whole $\Lambda$.
\end{proof}
\subsubsection{Link to the theory of fibered packings}\label{fiberedpackingsect}
\noindent In this section we consider our Theorem \ref{thm11} within the theory of fibered packings, as introduced by Conway and Sloane \cite{csbestpack} and extended by Cohn and Kumar \cite{cohnkumarsoft}. The basic definition is the following:
\begin{defi}[fibering configurations, cf. \cite{csbestpack}]\label{deffibering}
Let $\Lambda\subset\mathbb R^n, \Lambda_0\subset \mathbb R^m, m<n$ be discrete configurations of points. We say that $\Lambda$ fibers over $\Lambda_0$ if $\Lambda$ can be written as a disjoint union of layers belonging to parallel $m$-planes each of which is isometric to $\Lambda_0$.
\end{defi}
\noindent In fact the configurations considered in \cite{csbestpack} and \cite{cohnkumarsoft} are of a more special type, as described below.

\begin{defi}[lattice-periodic fibered configurations]\label{defperiodicfiber}
Let $\Lambda_1\subset\{0\}\times\mathbb R^{n-m}, \Lambda_0\subset \mathbb R^m\times\{0\}$ be Bravais lattices. Let $H\subset\mathbb R^m\times\{0\}$ be a set of cardinality $|H|\le m$ having the same symmetries as $\Lambda_0$ in the sense of Definition \ref{samesymmetries}. Given a periodic map $\mathbf s:\Lambda_1\to H$ we define the configuration 
\begin{equation}\label{latticeperfiber}
\Lambda_{\mathbf s}:=\cup_{p\in\Lambda_1}(\Lambda_0+ \mathbf s(p))\times\{p\}.
\end{equation}
\end{defi}
Recall that $\mathbf s:\Lambda_1\to\mathbb R^m$ is said to be \emph{$\Lambda_2$-periodic} if $\Lambda_2$ is a sublattice of $\Lambda_1$ and $\mathbf s(a)=\mathbf s(b)$ whenever $a-b\in\Lambda_2$.\\
Note that the definition \eqref{defls2} is a special case of the above definition for $\Lambda_1= \tau\Z$.

\medskip
\noindent  All configurations considered in \cite{csbestpack} and \cite{cohnkumarsoft} are of the above form, with $H$ consisting precisely of the origin and the so-called \emph{deep holes} of $\Lambda_0$, and $\Lambda_0$ is always either equal to the $2$-dimensional triangular lattice $A_2$ or to the $4$-dimensional lattice $D_4$ or to the $8$-dimensional lattice $E_8$. Recall (see \cite[Defn. 1.8.4]{Martinet}) that $p\in\mathbb R^m$ is a deep hole of $\Lambda_0\subset \mathbb R^m$ if it realizes the maximum of $p\mapsto\min_{q\in\Lambda_0}|p-q|$. The set of deep holes verifies \eqref{eqsamesymmetries} in the cases $\Lambda_0\in\{A_2,D_4.E_8\}$: for the triangular lattice $A_2$ the deep holes are the centers of the fundamental triangles, and thus the isometries of the lattice are transitive on the deep holes; for $E_8$ the deep holes are $1/2 E_8$ and the statement is again clear; for the rescaled version $D_4=\{(x_1,\ldots,x_4)\in\mathbb Z^4: \sum_ix_i\equiv 0 (\text{mod} 2)\}$ the deep holes within the fundamental cell are $(1,0,0,0), (1/2,1/2,1/2,1/2), (1/2,1/2,1/2,-1/2)$, which again are equivalent under symmetries of $D_4$. The following questions are worth mentioning at this point:

\begin{question}
Is it true that for any Bravais lattice $\Lambda_0$ the set of deep holes, i.e. the set of maximizers of $p\mapsto \min_{q\in\Lambda_0}|p-q|$, have the same symmetries as $\Lambda_0$ in the sense of Definition \ref{samesymmetries}? Is this the case for deep holes of homogeneous discrete $\Lambda\subset \mathbb R^{d-1}$? 
\end{question}

\noindent If $\mathbf s:\Lambda_1\to H$ and $\Lambda$ are like in Definition \ref{defperiodicfiber}, $E_{f,\Lambda}(x)$ is defined like in \eqref{defeflambdax} and $\mathbf s$ is $\Lambda_2$-periodic and $r(\Lambda_1/\Lambda_2)$ is a set of representatives in $\Lambda_1$ of $\Lambda_1/\Lambda_2$ (i.e. the discrete version of a fundamental domain), then in case $|\Lambda_1:\Lambda_2|<\infty$ we define the energy per point with the following formula generalizing \eqref{defefl}:
\begin{equation}\label{defeflgen}
E^f(\Lambda_{\mathbf s}):=\frac{1}{|\Lambda_1:\Lambda_2|} \sum_{k\in r(\Lambda_1/\Lambda_2)}\sum_{h\in\Lambda_1}E_{f,\Lambda_0}(\mathbf s(k)-\mathbf s(h) + h-k).
\end{equation}
For the case $f(r)=f_\alpha(r)=e^{-\pi\alpha r}$ we have the following definition:
\begin{defi}[asymptotic minimizers]\label{defasymin}
Let $H$ and $\Lambda_0,\Lambda_1$ be fixed and like in Definition \ref{defperiodicfiber}. We say that a periodic $\mathbf s:\Lambda_1\to H$ is \emph{asymptotically minimizing} as $\alpha\to\alpha_0\in[0,+\infty]$ if for any $\epsilon>0$ and any periodic $\mathbf s':\Lambda_1\to H$ with $\mathbf s'\neq \mathbf s$ there exists a neighborhood $\mathcal N$ of $\alpha_0$ such that for $\alpha\in \mathcal N$ we have 
\[
E^{f_\alpha}(\Lambda_{\mathbf s})\le \epsilon + E^{f_\alpha}(\Lambda_{\mathbf s'}).
\]
\end{defi}
\noindent We now can state in a new unified form the underlying reasoning subsuming the papers \cite{csbestpack} and \cite{cohnkumarsoft}:
\begin{thm}[asymptotics of periodic minimizers \cite{csbestpack,cohnkumarsoft}]
With the notations of Definition \ref{defasymin}, for $p\in\Lambda_1$ let $C_k(p)$ be the $k$-th layer of $\Lambda_1$ centered at $p$, defined for $k\ge 0$ as 
\[
C_0(p):= \{p\},\quad C_{k+1}(p):=\arg\min\left\{|q-p|:\ q\in\Lambda_1\setminus\cup_{j\le k} C_j\right\}.
\]
Let $C_k:=C_k(0)$ and 
\[
m_k(\mathbf s):=\frac{1}{|\Lambda_1:\Lambda_2||C_k|}\left|\left\{(p,q)\in r(\Lambda_1/\Lambda_2)\times \Lambda_1:\ q\in C_k(p), \mathbf s(q)\neq \mathbf s(p)\right\}\right|.
\]
Then the $\Lambda_2$-periodic configuration $\mathbf s$ is asymptotically minimizing as $\alpha\to+\infty$ if and only if the following infinite series of conditions $A(k)$ hold for all $k\ge 1$:
\begin{equation}\label{asymin_a1}
\text{A(1): }\quad \mathbf s \text{ realizes }\bar m_1:=\max\left\{m_1(\mathbf s'):\ \mathbf s':\Lambda_1\to H\text{ periodic}\right\}, 
\end{equation}
and for $k\ge 1$
\begin{equation}\label{asymin_ak}
\text{A(k+1): }\quad \mathbf s \text{ realizes }\bar m_{k+1}:=\max\left\{m_{k+1}(\mathbf s'):\ \mathbf s':\Lambda_1\to H\text{ periodic, and realizes }\bar m_k\right\}.
\end{equation}
\end{thm}
\noindent The proof of this result is presented in \cite{cohnkumarsoft} and consists in noticing that without loss of generality two separate $\mathbf s,\mathbf s'$ have the same period $\Lambda_2$, and then that for large $\alpha$ the contribution of $C_k(p)$ to $E^{f_\alpha}(\Lambda_{\mathbf s})$ becomes arbitrarily large compared to the combined one of all the $C_h(p)$ such that $h>k$.

\medskip
\noindent  The discussion of $A(1)$ in some special cases is the main topic of \cite{csbestpack}.

\medskip
\noindent  We note that if $\Lambda_1\sim \Z$ like in the previous subsection, we find the following rigidity result:
\begin{lemma}
If $\Lambda_1=\tau\Z$ and $|H|=d$ then the set of $\mathbf s$ satisfying all the A(k) coincides with the ones which have period $d$ and which realize a bijection to $H$ over each period. In particular conditions A(k) with $k\le d/2$ completely determine the optimal $\mathbf s$, and these $\mathbf s$ are uniquely defined up to composing with a permutation of $H$. 
\end{lemma}
\noindent The same type of rigidity (with a different bound on $k$) was discovered, with case-specific and often enumerative proofs, for the following couples of $(\Lambda_1,d)$ in \cite{csbestpack} and \cite{cohnkumarsoft}: $(A_2,d)$, $(D_3, d)$ and $(HCP, d)$ with $d\le 5$. By similar case-by-case computations we are able to prove the same results for $A_2$ for $d\le 8$ and for general two-dimensional lattices for $d\le 6$, as well as for $D_3$ for $d\le 6$. This leaves the following questions wide open, while giving strong evidence that the answer is positive:
\begin{question}[rigidity of the constraints $A(k)$]
Is it true that for every choice of $\Lambda_1$ and of $|H|$ there exists $h\ge 1$ such that the conditions A(k) with $k\ge h$ are redundant?
\end{question}
\begin{question}[uniqueness of the optimal $\mathbf s$]
Is it true that the $\mathbf s$ satisfying all the A(k) is always unique up to composition with permutations of $H$?
\end{question}

\section{Minimization of $(\Lambda,u)\mapsto \theta_{\Lambda+u}(\alpha)$}\label{minlu}

\subsection{Minimization in both $\Lambda$ and $u$}

\subsubsection{Upper bound for Regev-Stephens-Davidowitz function and consequences - Proof of Proposition \ref{prop12}}\label{UBRSD}

\begin{defi}
For any Bravais lattice $\Lambda\subset \R^d$, any $u\in \R^d$ and any $\alpha>0$, we define
$$
\rho_{\Lambda,u}(\alpha):=\frac{\theta_{\Lambda+u}(\alpha)}{\theta_\Lambda(\alpha)},
$$
where the theta functions are defined by \eqref{thetageneral} and \eqref{thetatranslated}.
\end{defi}
\begin{remark}
We remark that, in the terminology of Regev and Stephens-Davidowitz \cite{Regev:2015kq}, $\rho_{\Lambda,u}(\alpha)=f_{\Lambda,\alpha^{-2}}(u)$ where $f_{\Lambda,s}$ is the periodic Gaussian function over $\Lambda$ with parameter $s$.
\end{remark}
We restate \cite[Prop. 4.1]{Regev:2015kq} in terms of $\rho_{\Lambda,u}$.
\begin{lemma} \label{RSDrho} (Regev and Stephens-Davidowitz \cite{Regev:2015kq}) For any Bravais lattice $\Lambda\subset \R^d$ and any $u\in \R^d\backslash \Lambda$, $\alpha\mapsto \rho_{\Lambda,u}(\alpha)$ is a non-increasing function.
\end{lemma}
\noindent The above lemma is then complemented by the following independent result (which is well-known, and implicit in the work \cite{banasz}), that corresponds to the first part of Proposition \ref{prop12}:
\begin{prop}\label{translationregevstephens}
For any Bravais lattice $\Lambda\subset \R^d$, any $u\in \R^d$ and any $\alpha>0$, we have:
\begin{enumerate}
\item if $u\not\in \Lambda$, then $\displaystyle \lim_{\alpha \to +\infty} \rho_{\Lambda,u}(\alpha)=0$;
\item it holds $0<\rho_{\Lambda,u}(\alpha)\leq 1$, i.e. $\theta_{\Lambda+u}(\alpha)\leq \theta_\Lambda(\alpha)$. Furthermore, $\rho_{\Lambda,u}(\alpha)=1$ if and only if $u\in \Lambda$, i.e. for fixed $\Lambda$ and $\alpha$, the set of maximizers of $u\mapsto \rho_{\Lambda,u}(\alpha)$ (or $u\mapsto \theta_{\Lambda+u}(\alpha)$) is exactly $\Lambda$. 
\end{enumerate}
\end{prop}
\begin{proof}
If $u\not\in \Lambda$, then we obtain
$$
\rho_{\Lambda,u}(\alpha)=\frac{\theta_{\Lambda+u}(\alpha)}{\theta_\Lambda(\alpha)}=\frac{\theta_{\Lambda+u}(\alpha)}{\displaystyle 1+\sum_{p\in \Lambda\backslash \{0\}}e^{-\pi \alpha |p|^2}}
$$
which goes to $0$ as $\alpha \to +\infty$ because $p+u\neq 0$ for any $p\in \Lambda$ and any $u\not\in \Lambda$.\\
By Poisson summation formula (see for instance \cite[Thm. A]{SaffLongRange}), we have, for any $u\in\mathbb R^d$ and any $\alpha>0$,
$$
\theta_{\Lambda+u}(\alpha)=\frac{\alpha^{-d/2}}{|\Lambda|}\sum_{s\in \Lambda^*} e^{2 i\pi s\cdot u} e^{-\frac{\pi |s|^2}{\alpha}}.
$$
Hence we get
$$
\rho_{\Lambda,u}(\alpha)=\left |\rho_{\Lambda,u}(\alpha) \right|=\left|\frac{\theta_{\Lambda+u}(\alpha)}{\theta_\Lambda(\alpha)}\right| = \left| \frac{\displaystyle\sum_{s\in \Lambda^*} e^{2 i\pi s\cdot u} e^{-\frac{\pi |s|^2}{\alpha}}}{\displaystyle\sum_{s\in \Lambda^*}  e^{-\frac{\pi |s|^2}{\alpha}}}\right|\leq 1.
$$
Furthermore, we have, for any fixed $\Lambda,\alpha$:
\begin{align*}
\rho_{\Lambda,u}(\alpha)=1&\iff \theta_\Lambda(\alpha)=\theta_{\Lambda+u}(\alpha)\\
&\iff \sum_{s\in \Lambda^*} e^{-\frac{\pi |s|^2}{\alpha}}(1-\cos(2\pi s\cdot u))=0\\
&\iff \forall s\in \Lambda^*, 2\pi s\cdot u = 0\quad (mod \pi)\\
&\iff \forall s\in \Lambda^*, s\cdot 2u\in \Z \\
&\iff 2u \in \Lambda \\
& \iff u\in \Lambda.
\end{align*}
\end{proof}
\begin{remark}
In particular, for any $\Lambda$ and $\alpha>0$, if $u\not\in \Lambda$, then $\rho_{\Lambda,u}(\alpha)<1$.
\end{remark}
\noindent In dimension $d=2$, we get the second part of Proposition \ref{prop12} as a corollary of the previous result:
\begin{corollary}\label{cortrilattice}
Let $A_2$ be the triangular lattice of length $1$, then for any $\alpha>0$, any vector $u\in \mathbb R^2$ and any Bravais lattice $\Lambda$ such that $|\Lambda|=|A_2|$, it holds
$$
\theta_{A_2 +u}(\alpha)\leq \theta_\Lambda(\alpha),
$$
with equality if and only if $u\in A_2$ and $\Lambda=A_2$ up to rotation.
\end{corollary}

\begin{proof}
Let $u\in \mathbb R^2$, $\alpha>0$ and $\Lambda$ be a Bravais lattice of $\mathbb R^2$. By the previous proposition and Montgomery's Theorem \cite[Thm. 1]{Mont}, we get
$$
\theta_{A_2+u}(\alpha)\leq \theta_{A_2}(\alpha)\leq \theta_\Lambda(\alpha).
$$
The equality holds only if $\theta_{A_2+u}(\alpha)=\theta_{A_2}(\alpha)$ and $\theta_{A_2}(\alpha)=\theta_\Lambda(\alpha)$, i.e. respectively $u\in A_2$ and $\Lambda=A_2$ up to rotation.
\end{proof}

\noindent We now give an alternative proof of Lemma \ref{RSDrho} in the particular case $\Lambda=\Z$ and $u=1/2$ because it will be useful in the last part of this paper. We note that $\rho_{\Z,1/2}(\alpha)^2$ is also called the modulus of the elliptic functions (see \cite[Ch. 2]{Lawden}). We recall that the Jacobi theta functions $\theta_i$ are defined by \eqref{jacobithetadef}. 

\medskip
\noindent  We first prove the following. An alternative proof, found by Tom Price, is available online at \cite{proofpricemo}.
\begin{prop}\label{ellipticmodulus}
Let $\displaystyle \rho_{\Z,1/2}(\alpha)=\frac{\theta_2(\alpha)}{\theta_3(\alpha)}$, then
\begin{enumerate}
\item for any $\alpha>0$, $0<\rho_{\Z,1/2}(\alpha)<1$;
\item the function $\rho_{\Z,1/2}$ is decreasing on $(0,+\infty)$.
\end{enumerate}
\end{prop}
\begin{proof}
The first point is a direct application of Proposition \ref{translationregevstephens}, because $1/2\not\in \Z$. For the second point, we remark that (see \cite[Eq. (2.1.8)]{Lawden})
$$
\rho_{\Z,1/2}(\alpha)=(1-k'^2)^{1/4}
$$
where $\displaystyle k':=\frac{\theta_4(\alpha)^2}{\theta_3(\alpha)^2}\leq 1$ is the complementary modulus of the elliptic functions. For $q=e^{-\pi \alpha}$ we have, by the Jacobi's triple product formula \cite[Ch. 10, Thm. 1.3]{steinshaka},
$$
k'^2=\prod_{n=1}^{+\infty} \left( \frac{1-q^{2n-1}}{1+q^{2n-1}} \right)^8.
$$
All the factors are increasing in $\alpha$, we see that $\alpha\mapsto k'^2$ is an increasing function, and it follows that $\rho_{\Z,1/2}$ is decreasing.
\end{proof}
\noindent We now fix the decomposition $\mathbb R^d=\mathbb R^{d-1}\times\mathbb R$ and in these coordinates we consider the case where we are given two bounded functions $\mathbf t:\mathbb Z\to\{0\}\times\mathbb R$ and $\mathbf s:\mathbb Z\to \mathbb R^{d-1}\times \{0\}$ and the lattices obtained as translations of a Bravais lattice $\Lambda_0$, as follows:
\[
\Lambda_{\mathbf s,\mathbf t,\Lambda_0}:=\bigcup_{k\in\mathbb Z}\left(\mathbf t(k)+\mathbf s(k)+\Lambda_0\right),\quad \Lambda_{\mathbf s\equiv 0,\mathbf t,\Lambda_0}:=\bigcup_{k\in\mathbb Z}\left(\mathbf t(k)+\Lambda_0\right).
\]
Then we see that as for $p\in \Lambda_0$ by the orthogonality of $\mathbf t(k)$ to $p,\mathbf s(k)$ there $|\mathbf t(k)+\mathbf s(k)+p|^2=|\mathbf t(k)|^2+|\mathbf s(k)+p|^2$ so we can factorize:
\[
\theta_{\Lambda_{\mathbf s\equiv 0,\mathbf t,\Lambda_0}}(\alpha)-\theta_{\Lambda_{\mathbf s,\mathbf t,\Lambda_0}}(\alpha) = \frac{\alpha^{-d/2}}{|\Lambda_0|}\sum_{k\in\mathbb Z} e^{-\pi\alpha|\mathbf t(k)|^2}\left(\sum_{q\in \Lambda_0^*} e^{-\frac{\pi |q|^2}{\alpha}}-\sum_{q\in \Lambda_0^*} e^{2 i\pi q\cdot \mathbf s(k)}e^{-\frac{\pi |q|^2}{\alpha}}\right), 
\]
and using again the positivity and monotonicity of each of the terms in parentheses we deduce that a choice of ``horizontal'' translations $\mathbf s(k)$ which maximises the energy $\theta_{\Lambda_{\mathbf s,\mathbf t,\Lambda_0}}(\alpha)$ is the one given by $\mathbf s(k)=0$ for all $k\in\mathbb Z$, which is the case of perfectly aligned copies of $\Lambda_0$.

\begin{corollary}
For any $\alpha>0$, any $\mathbf s,\mathbf t$ as above and any Bravais lattice $\Lambda_0$ in $\mathbb R^{d-1}\times\{0\}$, it holds
$$
\theta_{\Lambda_{\mathbf s,\mathbf t,\Lambda_0}}(\alpha)\leq \theta_{\Lambda_{\mathbf s\equiv 0,\mathbf t,\Lambda_0}}(\alpha).
$$
\end{corollary}

\subsubsection{The degeneracy of Gaussian energy}\label{secdegeneracy}
We now show that the formula
\begin{equation}\label{limitalpha0}
\lim_{\alpha\to 0^+}\frac{\theta_{\Lambda + u}(\alpha)}{\theta_{\Lambda}(\alpha)} =1
\end{equation}
holds more in general, even when the perturbation $u$ depends on the point. We note that in higher generality the monotonicity in $\alpha>0$ of the above ratio is unknown, therefore we cannot extend the results of the previous section.
 \begin{prop}
 Let $\Lambda_0\subset\mathbb R^d$ be a lattice of determinant one and let $\mathbf u:\Lambda_0\to \mathbb R^d$ be a bounded function and $\Lambda_{\mathbf u}=\{\mathbf u(p)+p:\ p\in\Lambda_0\}$. Then there holds 
 \begin{equation}
 \label{alphato0}
 \lim_{\alpha\to 0}\frac{\theta_{\Lambda_{\mathbf u}}(\alpha)}{\theta_{\Lambda_0}(\alpha)} = 1.
 \end{equation}
 \end{prop}
 \begin{proof}
 We may write 
 \[
 \frac{\theta_{\Lambda_{\mathbf u}}(\alpha)}{\theta_{\Lambda_0}(\alpha)} = \frac{\sum_{p\in\Lambda_0}e^{-\alpha|\mathbf u(p)+p|^2}}{\sum_{p\in\Lambda_0}e^{-\alpha|p|^2}}.
 \]
Since $u$ is bounded, we have, for some constant $C>0$,
\begin{equation}\label{policemen}
e^{-\pi\alpha(|p|-C)^2}\le e^{-\pi\alpha|\mathbf u(p)+p|^2}\le e^{-\pi\alpha(|p|+C)^2}.
\end{equation}
Since $\Lambda_0$ is a lattice of determinant one we have 
\begin{eqnarray*}
\sum_{p\in\Lambda_0}e^{-\pi\alpha|p|^2}&=&  (1+ o(1)_{\alpha\to 0})\int_{\mathbb R^d}e^{-\pi\alpha|x|^2} dx,\\
\sum_{p\in\Lambda_0}e^{-\pi\alpha(|p|-C)^2}&=& (1+ o(1)_{\alpha\to 0})\int_{\mathbb R^d}e^{-\pi\alpha(|x|-C)^2} dx,\\
\sum_{p\in\Lambda_0}e^{-\pi\alpha(|p|+C)^2}&=& (1+ o(1)_{\alpha\to 0})\int_{\mathbb R^d}e^{-\pi\alpha(|x|+C)^2} dx,\\
\end{eqnarray*}
Thus we have to show that the ratio of the last two integrals on the left hand side tends to $1$ as $\alpha\to 0$. Note that the numerator of this ratio has limit $+\infty$ as $\alpha\to 0$. If $C_d=2\pi^{d/2}/\Gamma(d/2)$ is the area of the unit sphere in $\mathbb R^d$, after a change of variable we find
\[
\int_{\mathbb R^d}e^{-\pi\alpha(|x|-C)^2} dx=C_d\int_0^Ce^{-\pi\alpha(r-C)^2}r^{d-1} dr + C_d\int_0^\infty e^{-\pi\alpha r^2}(r+C)^{d-1} dr
\]
and the first term is an $O(1)_{\alpha\to 0}$, and therefore this term will thus disappear in the limit. For the second term we note that for any $0\le a<b$ there holds
\[
\int_0^\infty e^{-\pi\alpha r^2}r^a dr\ll\int_0^\infty e^{-\pi\alpha r^2}r^b dr\ \text{ as }\alpha\to 0,
\]
therefore upon expanding the polynomial $(r+C)^{d-1}$ the term in $r^{d-1}$ is the leading order term. By an analogous reasoning for the case of $e^{-\pi\alpha(|p|+C)^2}$ we deduce that the ratio of sums of the leftmost and rightmost terms in \eqref{policemen} have limits which agree and equal $1$, and the claim follows.
\end{proof}

\subsubsection{Iwasawa decomposition and reduction to diagonal matrices - Proof of Proposition \ref{prop13} (first part)}\label{secIwasawa}
\noindent In this subsection, we prove the first part of Proposition \ref{prop13}.
We recall that Bravais lattices of density one in $\mathbb R^d$ are precisely the lattices $\Lambda= M\mathbb Z^d$ for $M\in SL(d)$. By Iwasawa decomposition of $SL(d)$ any such $M$ can be expressed in the form 
\begin{equation}\label{iwasawa}
M=QDT,\ \text{ with }\ \left\{\begin{array}{l}Q\in SO(d),\\ D=\text{diag}(c_1,\ldots,c_d),\ c_i> 0,\ \prod_{i=1}^dc_i=1,\\
T\text{ lower triangular.}
\end{array}\right.
\end{equation}
\begin{proof}[Proof of Proposition \ref{prop13} - First part]
For any Bravais lattice  $\Lambda=M\mathbb Z^d$ with $M\in SL(d)$ decomposed as $M=QDT$ with notations like in \eqref{iwasawa}, any $u\in\mathbb R^d$ and any $\alpha>0$, let us prove that   
$$
\theta_{\Lambda+u}(\alpha)\geq \theta_{D(\Z +1/2)^d}(\alpha) .
$$
First of all we note that for any $Q\in SO(d)$ there holds
\[
\theta_{Q^{-1}\Lambda}(\alpha)=\theta_\Lambda(\alpha).
\]
Thus we may as well assume that in \eqref{iwasawa} we have $Q=id$. As $M$ is invertible, we may express $u=Mv$, thus
\[
\theta_{\Lambda+u}(\alpha)=\theta_{DT(\mathbb Z^d+v)}(\alpha).
\]
Next, we use the parametrization similar to \cite[p. 76]{Mont} of Montgomery in order to re-express for $p=(n_1,\ldots,n_d)\in\mathbb Z^d$ and writing $v=(v_1,\ldots,v_d)$, $T=T_{ij}$ with $T_{ij}=0$ for $j>i$ and $T_{ii}=1$,
\begin{eqnarray*}
|DT(p+v)|^2&=& \sum_{i=1}^d\left|c_i(n_i+v_i)+ c_i\sum_{j=1}^{i-1}T_{ij}(n_j+v_j)\right|^2\\
&:=&\sum_{i=1}^d\left|c_i n_i+ F_i(n_1,\ldots,n_{i-1})\right|^2,
\end{eqnarray*}
where $F_i$ depends only on $T,D,v$ and $F_1=0$. Thus, we have that $\theta_{DT(\mathbb Z^d+v)}(\alpha)$ equals
\begin{equation}\label{nestedsum}
\sum_{n_1\in \Z} e^{-\pi \alpha c_1^2| n_1|^2}\left(\sum_{n_2\in\Z}e^{-\pi \alpha c_2^2|n_2+ c_2^{-1}F_2(n_1)|^2}\left(\cdots\left(\sum_{n_d\in\Z}e^{-\pi \alpha c_d^2| n_d+ c_d^{-1}F_d(n_1,\ldots,n_{d-1})|^2}\right)\cdots\right)\right).
\end{equation}
Now, we know that (see \cite[Lem. 1]{Mont}), for any $\beta>0,s\in\mathbb R$,
$$
\sum_{m\in \Z} e^{-\pi \beta(m+s)^2}\geq \sum_{m\in \Z} e^{-\pi \beta (m+1/2)^2}.
$$
Applying this for $m=n_d,\beta=\alpha c_d^2$ and $s=c_d^{-1}F_d(n_1,\ldots,n_{d-1})$ we obtain that the innermost sum in \eqref{nestedsum} satisfies 
\[
\sum_{n_d\in\Z}e^{-\pi \alpha c_d^2| n_d+ c_d^{-1}F_d(n_1,\ldots,n_{d-1})|^2}\ge  \sum_{n_d\in \Z} e^{-\pi \alpha c_d^2(n_d+1/2)^2} =\theta_{c_d(\mathbb Z +1/2)}(\alpha),
\]
and therefore the latter sum can be brought outside of the nested parentheses in \eqref{nestedsum} and we obtain that 
\[
\theta_{DT(\Z^d+v)}(\alpha)\geq \theta_{c_d(\Z+1/2)}(\alpha) \theta_{D'T'(\Z^{d-1}+v')}(\alpha),
\]
where the $(d-1)\times(d-1)$-matrices $T',D'$ are obtained from $T,D$ by forgetting the last row and column, and $v'=(v_1,\ldots,v_{d-1})$. Using this we can easily prove by induction on $d$ that 
\[
\theta_{DT(\Z^d+v)}(\alpha)\ge\prod_{i=1}^d\theta_{c_i(\Z+1/2)}(\alpha),
\]
and by the defining property $e^{a+b}=e^ae^b$ and by distributivity (which can be applied here due to the fact that our sums defining $\theta$ functions are all absolutely convergent), the left hand side above is just $\theta_{D(\Z + 1/2)^d}(\alpha)$, as desired.
\end{proof}

\subsubsection{Discussion about $M\mapsto \theta_{M(\mathbb Z+1/2)^d}(\alpha)$ for diagonal $M$ - Proof of Proposition \ref{prop13} (second part)}\label{Discussdiag}
\noindent The outcome is a description of the minimization which complements the result of the first part of Proposition \ref{prop13}.\\
  We have the following:
 
\begin{prop}\label{theoremrectangular}
There exists a sequence $(\Z^2_k:=k\Z \times \frac{1}{k}\Z)_k$ of rectangular lattices with density $1$ such that 
$$
\lim_{k\to +\infty} \theta_{\Z^2_k +(k/2,1/2k)}(\alpha)=0.
$$
\end{prop}
\begin{proof}
Keeping the parametrization from \cite{Mont} for lattices with area $1/2$, for any $\alpha>0$, let 
$$
\displaystyle f_\alpha (y):=\theta_{\tilde{\Z}^2_y+(\frac{\sqrt{y}}{2\sqrt{2}},\frac{1}{2\sqrt{2y}})}(\alpha)=\sum_{n\in \Z} e^{-\frac{\pi \alpha}{2} y(n+1/2)^2}\sum_{m\in \Z} e^{-\frac{\pi \alpha}{2y}(m+1/2)^2},
$$
where $\tilde{\Z}^2_y=\frac{\sqrt{y}}{\sqrt{2}}\Z \times \frac{1}{\sqrt{2y}}\Z$ is the rectangular lattice of area $1/2$ parametrized by $y\geq 1$. Therefore, using notations \eqref{jacobithetadef} and by \eqref{theta24id}, we get, for any $\alpha>0$,
\begin{align*}
f_{2\alpha}(y)&=\theta_2(\alpha y)\theta_2(\alpha/y)= \sqrt{\frac{y}{\alpha}}\theta_4(y/\alpha)\theta_2(\alpha y)\\
&=\sqrt{\frac{y}{\alpha}} \left( 1+2\sum_{n=1}^{+\infty} (-1)^n e^{-\pi y n^2/\alpha} \right)\sum_{n\in \Z}e^{-\pi\alpha y(n+1/2)^2}.
\end{align*}
Hence, by growth comparison, we get, for any $\alpha>0$, $\displaystyle \lim_{y\to +\infty} f_\alpha(y)=0$, and there exists a sequence $(\Z^2_k:=k\Z \times \frac{1}{k}\Z)_k$ of rectangular lattices such that 
$$
\lim_{k\to +\infty} \theta_{\Z^2_k +(k/2,1/2k)}(\alpha)=0.
$$
\end{proof} 
\noindent Let us now prove the second part of Proposition \ref{prop13}, i.e. the existence of a sequence $A_k=\text{diag}(1,\ldots,1,k,1/k)$, $k\geq 1$, of $d\times d$ matrices such that, for any $\alpha>0$,
$$
\lim_{k\to +\infty}\theta_{A_k(\Z+1/2)^d}(\alpha)=0.
$$ 
\begin{proof}[Proof of Proposition \ref{prop13} - Second part] 
For any $k>0$, let $\Z_k^d:=\Z^{d-2}\times \Z_k^2$, where
$$
\Z_k^2:=k\Z \times \frac{1}{k}\Z.
$$
The center of its primitive cell is given by $C_k:=(1/2,...,1/2,k/2,1/2k)$ and
$$
\theta_{A_k(\mathbb Z+1/2)^d}(\alpha) =\theta_{\Z_k^d+C_k}(\alpha)=\theta_2(\alpha)^{d-2} \theta_{\Z_k^2+c_k}(\alpha),
$$
where $c_k=(k/2,1/2k)$. Now, by previous Proposition \ref{theoremrectangular}, as 
$$
\lim_{k\to +\infty}\theta_{\Z_k^2+c_k}(\alpha)=0,
$$
we get our result.
\end{proof}

\subsection{Proof of Theorem \ref{thm14}}\label{proofthm14}

\noindent We recall the following result \cite[Thm. 2]{Mont} due to Montgomery about minimization of theta functions among orthorhombic lattices:

\begin{thm} \label{theta3Mgt} (Montgomery \cite{Mont}) 
Let $d\geq 1$ and $\alpha>0$. Assume that $\{c_i\}_{1\leq i\leq d}\in (0,+\infty)^d$ are such that $\prod_{i=1}^d c_i =1$ and that not all the $c_i$ are equal to $1$. For real $t$, we define
$$
U_3(t)=\prod_{i=1}^d \theta_3(c_i^t \alpha).
$$
Then $U_3'(0)=0$, $U_3'(t)<0$ for $t<0$ and $U_3'(t)>0$ for $t>0$. In particular, $t=0$ is the only strict minimum of $U_3$.
\end{thm}
\begin{remark}
For $d=2$, a particular case, also proved in \cite[Thm. 2.2]{Faulhuber:2016aa}, is the fact that $y=1$ is the only minimizer of $y\mapsto \theta_3(\alpha y)\theta(\alpha y^{-1})$.
\end{remark}

We now prove Theorem \ref{thm14} that generalizes \cite[Thm. 2.4]{Faulhuber:2016aa} to general dimensions, in the spirit of the previous result.

\begin{proof}[Proof of Theorem \ref{thm14}] Let $d\geq 1$ and $\alpha>0$. Assume that $\{c_i\}_{1\leq i\leq d}\in (0,+\infty)^d$ are such that $\prod_{i=1}^d c_i =1$ and that not all the $c_i$ are equal to $1$. For real $t$, we define
\begin{itemize}
\item $\displaystyle U_4(t)=\prod_{i=1}^d \theta_4(c_i^t \alpha)$,
\item $\displaystyle U_2(t)=\theta_{A_t(\Z+1/2)^d}(\alpha)=\prod_{i=1}^d \theta_2(c_i^t \alpha)$,
\item $\displaystyle P_{3,4}(t)=U_3(t)U_4(t)=\prod_{i=1}^d \theta_3(c_i^t \alpha)\theta_4(c_i^t \alpha)$,
\item $\displaystyle P_{2,3}(t)=U_2(t)U_3(t)=\prod_{i=1}^d \theta_2(c_i^t \alpha)\theta_3(c_i^t \alpha),$
\item $\displaystyle Q(t)=\frac{U_2(t)}{U_3(t)}=\frac{\theta_{A_t(\Z+1/2)^d}(\alpha)}{\theta_{A_t\Z^d}(\alpha)}=\prod_{i=1}^d \frac{\theta_2(c_i^t \alpha)}{\theta_3(c_i^t \alpha)}$.
\end{itemize}
where $A_t=diag(c_1^t,...,c_d^t)$. Let us prove that, for any $f\in \{ U_4, U_2, P_{3,4}, P_{2,3} \}$, we have $f'(0)=0$, $f'(t)>0$ for $t<0$ and $f'(t)<0$ for $t>0$ and, in particular, that $t=0$ is the only strict maximum of $f$.

\medskip

\noindent As in \cite[Sec. 6]{Mont}, we remark that
$$
\frac{U_4'}{U_4}(t)=\alpha\sum_{i=1}^d \frac{\theta_4'}{\theta_4}(c_i^t \alpha) c_i^t \log c_i \quad \text{and} \quad   \left( \frac{U_4'}{U_4} \right)'(t)=\sum_{i=1}^d T(c_i^t \alpha)(\log c_i)^2,
$$
where $T$ is defined on $(0,+\infty)$ by
$$
T(x):=x\frac{\theta_4'}{\theta_4}(x)+x^2 \left( \frac{\theta_4'}{\theta_4} \right)'(x).
$$
Hence, we have by \cite[Thm. 2.2]{Faulhuber:2016aa}, for any $x>0$,
$$
T(x)=x\frac{\theta_4'}{\theta_4}(x)+x^2\frac{\theta_4''\theta_4-(\theta_4')^2}{\theta_4^2}(x)<x\frac{\theta_4'}{\theta_4}(x)-x^2\frac{\theta_4'(x) \theta_4(x)}{x\theta_4(x)^2}=0.
$$
Therefore, we get $\displaystyle \left( \frac{U_4'}{U_4} \right)'(t)<0$ for any $t$, and it follows that $\displaystyle \frac{U'}{U}$ is strictly decreasing. Thus, as
$$
\frac{U_4'}{U_4}(0)=\alpha \frac{\theta_4'(\alpha)}{\theta_4(\alpha)}\sum_{i=1}^d \log c_i=0,
$$
we get the result for $f=U_4$, because $U_4(t)>0$ for any $t$.

\medskip
\noindent  The case $f=U_2$ is a direct application of the previous point, thanks to identity \eqref{theta24id}. Indeed,
$$
U_2(t)=\prod_{i=1}^d \frac{\theta_4\left(\frac{1}{c_i^t \alpha}  \right)}{\sqrt{c_i^t \alpha}}=\alpha^{-1/2}\prod_{i=1}^d \theta_4\left(\frac{1}{c_i^t \alpha}  \right)
$$
because $\prod_{i=1}^d \sqrt{c_i^t}=\sqrt{\prod_{i=1}^d c_i^t}=1$. Now, writing, for any $i$, $\displaystyle b_i=\frac{1}{c_i}$, which are not all equal to $1$, and $\displaystyle \beta=\frac{1}{\alpha}$, we get $\prod_{i=1}^d b_i=1$ and
$$
U_2(t)=\sqrt{\beta}\prod_{i=1}^d  \theta_4(b_i^t \beta),
$$
which is exactly $\sqrt{\beta}U_4$ as in the previous point, therefore its variation is the same.

\medskip
\noindent  For $f=P_{3,4}$, by $\sqrt{\theta_3(s)\theta_4(s)}=\theta_4(2s)$ (see \cite[Sec. 4.1]{ConSloanPacking}) for any $s>0$, we get
$$
P_{3,4}(t)=\prod_{i=1}^d \theta_3(c_i^t \alpha')\theta_4(c_i^t \alpha')=\left(\prod_{i=1}^d \theta_4(2\alpha' c_i^t)  \right)^2=U_4(t)^2
$$
with $\alpha=2\alpha'$. Furthermore, for $f=P_{2,3}$, by 
$$
\theta_2(2s)=\sqrt{\frac{\theta_3(s)^2-\theta_4(s)^2}{2}},\theta_3(2s)=\sqrt{\frac{\theta_3(s)^2+\theta_4(s)^2}{2}} \quad \text{and}\quad  \theta_3(s)^4=\theta_2(s)^4+\theta_4(s)^4,
$$
(see \cite[Sec. 4.1]{ConSloanPacking}) we obtain
\begin{align*}
P_{2,3}(t)=\prod_{i=1}^d \theta_2(c_i^t \alpha')\theta_3(c_i^t \alpha')&=\frac{1}{2^d}\prod_{i=1}^d \sqrt{\theta_3(c_i^t \alpha')^4-\theta_4(c_i^t \alpha')^4}\\
&=\frac{1}{2^d}\prod_{i=1}^d \theta_2\left(  \frac{c_i^t \alpha'}{2}\right)^2\\
&= \frac{1}{2^d}U_2(t)^2,
\end{align*}
with $\displaystyle \alpha=\frac{\alpha'}{2}$. Hence, the two last points follow from the two first points.

\medskip

\noindent By using Theorem \ref{theta3Mgt} and the previous results we directly get the optimality for $Q$.
\end{proof}

\subsection{Minimization in $u$ at fixed $\Lambda$}\label{secfixedL}
\noindent The main motivation for the study in this subsection is the negative result of the first part of Theorem \ref{theoremrectangular}, which shows that if we minimize $\theta_{\Lambda+u}(\alpha)$ for varying $\Lambda$ and $u$, in general no minimizer exists.

\subsubsection{Two particular cases: orthorhombic and triangular lattices}
\noindent We take some time here in order to point out some special results about the minimization of $u\mapsto \theta_{\Lambda+u}(\alpha)$ valid for all $\alpha>0$ when $\Lambda$ is a special lattice.

\medskip
\noindent  The first result is just a special case of Proposition \ref{prop13}, and holds in general dimension:
\begin{corollary}[of Prop. \ref{prop13}]\label{coriwasawa}
Let $\alpha>0$ be arbitrary. If $\Lambda\subset \mathbb R^d$ is an orthorhombic lattice, i.e. $\Lambda=A\mathbb Z$ with $A$ a diagonal matrix, then the minimum of $u\mapsto \theta_{\Lambda+u}(\alpha)$ is achieved precisely at the points $\Lambda+c$ with $c=A\cdot (1/2,\ldots,1/2)$.
\end{corollary}
\noindent The second result, due to Baernstein, and is valid in two dimensions:
\begin{thm}[See {\cite[Thm. 1]{baernstein}}]\label{baernstthm}
Let $\alpha>0$ be arbitrary. 
If $\Lambda=A_2\subset \mathbb R^2$ is the triangular lattice of length $1$ generated by $(1,0),(1/2,\sqrt3/2)$, then the minimum of $u\mapsto \theta_{A_2+u}(\alpha)$ is achieved precisely at the points $A_2+c$ with $c=(1/2,1/(2\sqrt3))$.
\end{thm}
\noindent In particular, $c$ is the barycenter of the primitive triangle $\mathcal{T}:=\{(0,0);(1,0);(1/2,\sqrt{3}/2)  \}$. We are not aware of the presence of further results valid for all $\alpha>0$ in the literature.

\subsubsection{Convergence to a local problem in the limit $\alpha\to+\infty$ - Proof of Theorem \ref{thm15}}\label{secproof15}
\noindent In this part, we prove Theorem \ref{thm15} which expresses the fact that the minimization of $u\mapsto \theta_{\Lambda+u}(\alpha)$ has as $\alpha\to +\infty$ the same result as the maximization of $u\mapsto dist(u,\Lambda)$.
%
%
%
%
\begin{proof}[Proof of Theorem \ref{thm15}]
Let $\Lambda$ be a Bravais lattice in $\mathbb R^d$ and let $c$ be a solution to the following optimization problem:
\begin{equation}\label{maxneighbordist}
\max_{c'\in\mathbb R^d}\min_{p\in\Lambda}|c'-p|.
\end{equation}
For any $x\in\mathbb R^d$ let us prove that there exists $\alpha_x$ such that for any $\alpha>\alpha_x$, 
\begin{equation}\label{minimumc}
\theta_{\Lambda+c}(\alpha)\leq \theta_{\Lambda+x}(\alpha).
\end{equation}

\noindent Writing $\Lambda^\alpha:=\sqrt\alpha\Lambda$ we have
\begin{align*}
\theta_{\Lambda+c}(\alpha)&=\sum_{p\in \Lambda^\alpha} e^{-\pi |p+c\sqrt{\alpha}|^2}\\
&= \sum_{p\in \Lambda^\alpha, |p+c\sqrt{\alpha}|\leq \sqrt{d\gamma}}e^{-\pi |p+c\sqrt{\alpha}|^2} + \sum_{p\in \Lambda^\alpha, |p+c\sqrt{\alpha}|>\sqrt{d\gamma}}e^{-\pi |p+c\sqrt{\alpha}|^2}.
\end{align*}
By \cite[Lem. 18.2]{Barvinok}, we have that for $\gamma >1/2\pi$,
$$
 \sum_{p\in\Lambda^\alpha, |p+c\sqrt{\alpha}|>\sqrt{d\gamma}}e^{-\pi |p+c\sqrt{\alpha}|^2}\leq (2\pi\gamma)^{\frac{d}{2}} e^{-d\pi\gamma+d/2}\theta_{\Lambda_0}(\alpha).
$$
Moreover, for any $x\in \R^2$, we have (see \cite[Pb. 18.4]{Barvinok})
$$
\theta_{\Lambda}(\alpha)\leq e^{\pi \alpha |x|^2}\theta_{\Lambda+x}(\alpha).
$$
Thus, for any $x\in \R^2$, we get
\begin{eqnarray}
\theta_{\Lambda+c}(\alpha)&\leq& \sum_{p\in\Lambda, |p+c|\leq \sqrt{d\gamma/\alpha}}e^{-\pi\alpha |p+c|^2} +(2\pi\gamma)^{\frac{d}{2}} e^{-d\pi\gamma+d/2} \theta_{\Lambda}(\alpha)\nonumber\\
&\leq&\sum_{|p+c|\leq \sqrt{d\gamma/\alpha}}e^{-\pi |p+c\sqrt{\alpha}|^2} +(2\pi\gamma)^{\frac{d}{2}} e^{-d\pi\gamma+d/2}e^{\pi\alpha |x|^2}\theta_{\Lambda+x}(\alpha).\label{factor1}
\end{eqnarray}
Taking $\alpha$ so large that $\frac{\alpha |u|^2}{d}\ge \frac{1}{2\pi}$ and $\gamma:=\frac{\alpha |u|^2}{d}$ we have 
$$
\{p\in\Lambda:\ |p+c|\leq \sqrt{\frac{d\gamma}{\alpha}}\}=\emptyset.
$$
Furthermore we have that the factor in \eqref{factor1} rewrites as
\[
(2\pi\gamma)^{\frac{d}{2}} e^{-d\pi\gamma+d/2}e^{\pi\alpha |x|^2}=e^{-\pi\alpha(|c|^2-|x|^2)}\left(\frac{2\pi\alpha}{d}\right)^{\frac{d}{2}}e^{d/2}|c|^d.
\]
Therefore if $|x|<|c|$ then there exists $\alpha_x$ such that for any $\alpha>\alpha_x$, the inequality \eqref{minimumc} holds.

\medskip
\noindent  Now notice that unless $x$ also solves \eqref{maxneighbordist} there exist $p,q\in \Lambda$ such that $|p-x|<|q-c|$. In this case, we note that $\theta_{\Lambda + x}=\theta_{\Lambda +p+x}$ and thus we reduce to the case $|x|<|c|$, as desired.
\end{proof}
\noindent Note that in the case of $2$-dimensional Bravais lattices for which a fundamental domain is the union of two triangles with angles $\le \pi/2$ the points $c$ solving \eqref{maxneighbordist} are the circumcentres of the two triangles. For a rectangular lattice we find that $c$ is the center of the rectangle, while for the equilateral triangular lattice $A_2$ we have that $c=(1/2,1/2\sqrt3)$ is a minimizer for $\alpha\to +\infty$. These points are precisely the ones figuring also in Section \ref{baernstein}.

\subsubsection{Asymptotics of $\alpha\to 0$ by Poisson summation for $d=2$ - Proof of Theorem \ref{thm16}}\label{secproof16}
\noindent Let $\Lambda$ be a Bravais lattice in $2$ dimensions and $x\in \R^2$, then we have, for any $\alpha>0$, by the Poisson summation formula:
$$
\theta_{\Lambda+x}(\alpha)=\frac{1}{\alpha |\Lambda|}\sum_{\ell \in \Lambda^*} e^{-\frac{\pi}{\alpha}|\ell|^2 +2i \pi \langle \ell,x\rangle }.
$$
Thus we get
\begin{equation}\label{poissontri}
\theta_{\Lambda+x}\left(\frac{1}{\alpha} \right)=\alpha|\Lambda|^{-1} \sum_{\ell \in \Lambda^*}e^{-\pi \alpha |\ell |^2} e^{2i \pi \langle\ell,x\rangle}=\alpha|\Lambda|^{-1} \sum_{\ell \in \Lambda^*}e^{-\pi \alpha |\ell |^2}\cos(\pi \langle\ell,2x\rangle).
\end{equation}
Now, up to a rotation and reflections we may suppose that the matrix basis $B$ of $\Lambda$ is in the form
\[
B=\left(\begin{array}{cc}
a&b\\0&a^{-1}
\end{array}\right),\quad \text{for }a\ge 1, b\ge 0.
\]
Then the matrix of the basis of $\Lambda^*$ is $(B^T)^{-1}$. Therefore,
$$
\Lambda^*=\Z \left( a^{-1},-b\right) \oplus \Z \left( 0, a\right).
$$
Let $\ell_{m,n}:=m\left( a,-b\right)+ n\left( 0, a^{-1}\right)$ be a point of $\Lambda^*$, $x=s(a,0)+t\left(b, a^{-1}  \right)$ then we have, for any $(m,n)\in\Z^2$ and any $(s,t)\in [0,1/2]^2$,
$$
\langle \ell_{m,n}, 2x\rangle=2sm + 2tn.
$$
For the first layer $C_1$ of $\Lambda^*$, that is to say for any $\ell\in \Lambda^*$ such that $\displaystyle |\ell|^2=\min_{\ell'\in\Lambda^*\setminus\{0\}}|\ell'|^2$, we get
\begin{equation}\label{sumfirstlayer2d}
\sum_{(m,n)\in C_1}e^{-\pi \alpha |\ell_{m,n} |^2}\cos (\pi \langle \ell_{m,n}, 2x\rangle )=\sum_{(m,n)\in C_1}e^{-\pi \alpha |\ell_{m,n} |^2}\cos(2\pi(ms+nt)).
\end{equation}
We now aim at proving a result about the asymptotics of best translation vectors $x$, as $\alpha\to 0$. We use the following definition:
\begin{defi}\label{defasymintransl}
We say that $C$ is a set of \emph{asymptotic strict minimizers} of $x\mapsto \theta_{\Lambda+x}(\alpha)$ as $\alpha\to \alpha_0\in[0,\infty]$ if there exists $c\in C$ such that for each $x\notin C+\Lambda$ there exists a neighborhood of $\alpha_0$, denoted $U_x$, such that for all $\alpha\in U_x$ there holds 
\[
\theta_{\Lambda+c}(\alpha) < \theta_{\Lambda+ x}(\alpha).
\]
\end{defi}
\noindent Note that the above is very similar in spirit to Definition \ref{defasymin}, however here we minimize over the translation vectors $x$ while in Definition \ref{defasymin} we were minimizing over translation patterns $\mathbf s$.

\medskip

\noindent We now prove Theorem \ref{thm16}.

\begin{proof}[Proof of Theorem \ref{thm16}]
We have three options, corresponding to $C_1$ of cardinality either $6$ or $4$ or $2$, and to lattices $\Lambda$ of decreasing symmetry:
\begin{enumerate}
\item For the case $|C_1|=6$ we have that $\Lambda=\frac{2}{\sqrt{3}}A_2$ is a triangular lattice and $C_1$ corresponds to $\ell_{m,n}$ for $(m,n)\in\{\pm(1,0),\pm(0,1), \pm(1,1)\}$. In this case $|\ell|^2=4/3$ on the first layer and we compute
\begin{align*}
\text{\eqref{sumfirstlayer2d}}\ &=2e^{-\frac{4\pi}{3}\alpha}\left(\cos(2\pi s)+\cos(2\pi t)+\cos(2\pi(t+s))   \right)\\
&=2e^{-\frac{4\pi}{3}\alpha}\left(2\cos (\pi (t+s) ) \cos(\pi(t-s) )+2\cos^2(\pi(t+s))-1\right)\\
&=4e^{-\frac{4\pi}{3}\alpha}\cos (\pi (t+s) )\left(\cos (\pi (t+s) ) +\cos (\pi (t-s) )  \right) - 2e^{-\frac{4\pi}{3}\alpha}\\
&=8e^{-\frac{4\pi}{3}\alpha}\cos (\pi (t+s) )\cos(\pi t)\cos(\pi s) - 2e^{-\frac{4\pi}{3}\alpha}\geq - 3e^{-\frac{4\pi}{3}\alpha},
\end{align*}
with equality if and only if $s\in 1/3 + \mathbb Z$ or $t\in 1/3+\mathbb Z$.
\item The case $|C_1|=4$ corresponds to rhombic lattices different from $A_2$, in which up to change of basis of $\Lambda$ and $\Lambda^*$, $C_1$ has $(m,n)\in\{\pm(1,0), \pm(0,1)\}$ and similarly to above, we consider the sums like \eqref{sumfirstlayer2d} and to find optimal $(s,t)$ we have to minimize $\cos(2\pi s)+\cos(2\pi t)$ and thus we find that $s,t\in \mathbb Z+1/2$.
\item For $|C_1|=2$ we find precisely the non-rhombic lattices $\Lambda$ and up to changes of base of $\Lambda$ we find $C_1$ to have $(m,n)\in\{\pm(1,0)\}$. Here we have to minimize a formula like \eqref{sumfirstlayer2d} giving just $\cos(2\pi s)$ and thus we find $s\in \mathbb Z + 1/2$ but $t$ remains undetermined. To determine $t$ we have to look at the second layer, which again, can have either $6$, $4$ or $2$ lattice points:
\begin{enumerate}
\item If $|C_2|=2$ then up to a change of basis that keeps $C_1$ fixed we have $C_2=\{\ell_{m,n}:\ (m,n)\in\{\pm(0,1)\}\}$ and via the discussion of the sum like \eqref{sumfirstlayer2d} over the layer $C_2$ we directly find $t\in\mathbb Z+1/2$ precisely as above.
\item If $|C_2|=6$ then up to a change of basis keeping $C_1$ fixed we find $C_2$ has coefficients $(m,n)\in\{\pm(0,1), \pm(1,1), \pm(2,0)\}$, and we find that the minimization corresponding to \eqref{sumfirstlayer2d} for the second layer and with $s\in \mathbb Z+1/2$ gives $\cos(2\pi t)+\cos(2\pi (t+1/2))=0$, thus we have to discuss the third layer. Up to rotation and dilation $\Lambda^*$'s generators are $(1,0), (-1/2,\sqrt{15}/2)$. In this case we find that $C_3$ has coefficients $\{\pm(1,-1),\pm(2,1)\}$, $C_4$ has coefficients $\{\pm(2,0)\}$, giving no constraint on $t$, but $C_5$ has only coefficients $\{\pm(2,-1)\}$ which gives an expression uniquely minimized at $t\in \mathbb Z+1/2$.
\item If $|C_2|=4$ then up to a change of basis keeping $C_1$ fixed we find $C_2$ corresponds to $(m,n)\in\{\pm(0,1),\pm(1,1)\}$, which again gives no constraint on $t$. Up to rotation and dilation $\Lambda^*$'s generators are $(1,0), (-1/2,x)$ with $x\in]\sqrt 3/2,\sqrt{15}/2[$. We find several possibilities:
\begin{enumerate}
\item If $\sqrt 3/2<x<1$, then $C_3=\{\pm(1,-1), \pm(2,1)\}$, gives no constraint on $t$ and $C_4$ has coefficients $\{\pm(1,2)\}$ gives the minimization of $\cos(4\pi t+\pi)=-\cos(4\pi t)$ and so we have $t\in \frac12\Z+\tfrac14$. It seems that the remaining layers (we calculated them in some special cases till $C_{40}$ give no contribution allowing to decide whether there is a unique minimizer among $t\in \Z+\tfrac14$ and $t\in\Z+\tfrac34$, so we conjecture that they are both asymptotic minimizers.
\item If $x=1$, then $C_3$ is as before and $C_4$ has coefficients $\{\pm(1,2),\pm(2,0)\}$ and again  $t\in \frac12\Z+\tfrac14$. 
\item If $1<x<\sqrt 7/2$, then $C_3$ is as before, $C_4$ has coefficients $\{\pm(2,0)\}$ and $C_5$ has coefficients $\{\pm(1,2)\}$ is the deciding layer and we obtain again $t\in \frac12\Z+\tfrac14$.
\item If $x=\sqrt 7/2$ then $C_3$ has coefficients $\{\pm(1,-1), \pm(2,1),\pm(2,0)\}$ and again the deciding layer is $C_4$, which this time has coefficients $\{(\pm(1,2)\}$ and the minimizers are $t\in\tfrac12\Z + \tfrac14$.
\item If $\sqrt 7/2<x<5/(2\sqrt 3)$ then  $C_3$ has coefficients $\{\pm(2,0)\}, C_4=\{\pm(1,-1), \pm(2,1)\}$ and the decisive layer is $C_5$ has coefficients $\{\pm(1,2)\}$ and the minimizers are $t\in \frac12\Z+\tfrac14$.
\item If $x=5/(2\sqrt 3)$ then the first $4$ layers are as above, but $C_5$ has coefficients\\
 $\{\pm(1,2),\pm(2,-1),\pm(3,1)\}$, which again  leads to the minimizers $t\in \frac12\Z+\tfrac14$.
\item If $5/(2\sqrt 3)<x<3/2$ then the first four layers are as before, $C_5$ has coefficients $\{\pm(3,1),\pm(2,-1)\}$ gives zero contribution and $C_6$ has coefficients $\{\pm(1,2)\}$ gives minimizers $t\in \frac12\Z+\tfrac14$.
\item If $x=3/2$ then the first $5$ layers are as in the previous case but now $C_6$ has coefficients $\{\pm(1,2),\pm(3,0)\}$, still leading to $t\in \frac12\Z+\tfrac14$.
\item If $3/2<x<\sqrt{11}/2$ then the first $5$ layers are as in the previous case but now $C_6$ has coefficients $\{\pm(3,0)\}$ and $C_7$ has coefficients $\{\pm(1,2)\}$, still leading to $t\in \frac12\Z+\tfrac14$.
\item If $x=\sqrt{11}/2$ then the first $4$ layers are as in the previous case and $C_5$ has coefficients $\{\pm(2,-1), \pm(3,0),\pm(3,1)\}$ however the decisive layer is $C_6$, which has coefficients $\{\pm(1,2)\}$ and which gives  $t\in \frac12\Z+\tfrac14$. 
\item If $\sqrt{11}/2<x<\sqrt{15}/2$ then the first $4$ layers are as in the previous case but $C_5$ has coefficients $\{\pm(3,0)\}, C_6$ has coefficients $\{\pm(2,-1), \pm(3,1)\}$ and the decisive layer is $C_7$ which has coefficients $\{\pm(1,2)\}$, giving  $t\in \frac12\Z+\tfrac14$.
\end{enumerate}
\end{enumerate}
\end{enumerate} 
\noindent Now, if the quantity corresponding to the sum of the layers below or in the decisive layer is strictly larger than the minimum, then for $\alpha$ large enough the difference with the minimum will give a contribution which surpasses the contributions from all layers different than $C_1$, therefore by the above subdivision in cases we obtain the result. 
\end{proof}

\subsection{Questions and conjectures}\label{qcsec3}

\subsubsection{A question related to Proposition \ref{translationregevstephens}}

\noindent Note that the point 2. of Proposition \ref{translationregevstephens} may be sharp in the sense that superpositions of theta functions are the largest class where it keeps holding true in full generality. Indeed, consider, with the notation in \eqref{defeflambdax}, the quotient 
\begin{equation}\label{rholuf}
\rho_{\Lambda,u}(f):=\frac{E_{f,\Lambda}(u)}{E_{f,\Lambda}(0)}, 
\end{equation}
where $f:[0,+\infty)\to[0,+\infty)$ is a function more general than $f(r)=e^{-\pi\alpha r}$ used to define $\rho_{\Lambda,u}(\alpha)$ and let $F(x):= f(|x|^2)$. Then we have by the same reasoning as in the proposition, for each $u\in\mathbb R^d$,
\begin{equation}\label{sharpnessrholuf}
\rho_{\Lambda,u}(f)\le 1 \Leftrightarrow \left|\sum_{s\in \Lambda^*}\hat F (s)e^{2i\pi s\cdot u}\right|\le \left|\sum_{s\in \Lambda^*}\hat F(s)\right|.
\end{equation}
Requiring that the right side of \eqref{sharpnessrholuf} holds for all $u\in\mathbb R^d$ imposes on the $\hat F(s), s\in \Lambda^*$ the sharp condition, and is implied by the fact that $\hat F\ge 0$ on all rescalings of $\Lambda^*\setminus \{0\}$ or that the same holds for $-\hat F$. This property is equivalent to requiring that $\pm F$ is positive definite, or that $\pm f$ is completely monotone. Determining the class of interactions $F$ (or $f$) which is individuated by this condition seems to be an interesting open question:
\begin{question}\label{sharpnessquestion}
What is the class of all $F$ for which, for all $u\in\mathbb R^d$ the condition in the right-hand side of \eqref{sharpnessrholuf} holds?
\end{question}
\subsubsection{A conjecture related to Theorem \ref{thm16}}
\noindent \textbf{Conjecture:} \textit{For the Bravais lattices which are very asymmetric as in the last point above, both points in $C$ are separately realizing the minimum of $x\mapsto \theta_{\Lambda+x}(\alpha)$ as $\alpha\to 0$.}

\medskip
\noindent  This is not provable via examination of layers, but in several cases we explicitly observed all the layers of $\Lambda^*$ up to $C_{40}$ and we never found a contradiction to the conjecture. We thus have strong motivation to think that it is true. Perhaps it can be proved by symmetry considerations on the layers of $\Lambda^*$.

\section{Optimality and non optimality of BCC and FCC among body-centred-orthorhombic lattices - Proof of Theorems \ref{thm17} and \ref{thm18}}\label{secorthorombic}

\subsection{Proof of Theorem \ref{thm17}}\label{secproof17}
\noindent The goal of the present section is to prove Theorem \ref{thm17}, i.e. to study the minimization of $\Lambda\mapsto \theta_\Lambda(\alpha)$ among a class of three-dimensional Bravais lattices $\Lambda$ composed by stacking at equal distances a sequence of translated rectangular lattices with alternating translations (corresponding to a sequence of periodicity $2$ in Section \ref{seclayers}). This kind of systems was numerically studied by Ho and Mueller in \cite{Mueller:2002aa}. Indeed, in the context of two different Bose-Einstein condensates in interaction, in the periodic case, they proved that a good approximation of the energy of the system is given by
\begin{equation}\label{homuellereq}
E_\delta(\Lambda,u):=\theta_{\Lambda}(1)+\delta\theta_{\Lambda+u}(1),
\end{equation}
where $\Lambda\subset \R^2$ is a Bravais lattice, $u\in\R^2$ is the translation between the two copies of $\Lambda$ and $-1\leq \delta\leq 1$ is a parameter which quantifies the interaction between these two copies of $\Lambda$. While the case $\delta\leq 0$ gives $(A_2,0)$ as a minimizer of $E_\delta$ (i.e. the condensates are juxtaposed), varying $\delta$ from $0$ to $1$ gives a collection of minimizers composed by a triangular lattice that is deformed almost continuously to a rectangular lattice and where the translations correspond to the deep hole of each lattice (see \cite[Fig. 1]{Mueller:2002aa} or the review \cite[Fig. 16 and description]{ReviewvorticesBEC}). It is interesting to note that numerical simulations give strong support to the conjecture that also Wigner bilayers show the same behavior as Bose-Einstein condensates, as presented in \cite{Samaj12,TrizacWigner16}.

\medskip
\noindent  In the present section, we focus on the minimization in which $\Lambda$ is allowed to vary only in the class of rectangular lattices, but where $\theta_\Lambda(\alpha)$ is allowed to vary beyond the case $\alpha=1$, and for $\delta$ which is allowed to depend on $\alpha$ and on the distance $t$ between the layers of $\Lambda$ (compare the expression of the energy $E_t(y,\alpha)$ below to \eqref{homuellereq}). Also note that, by Corollary \ref{coriwasawa}, in the case of a rectangular lattice $\Lambda=L_y$ (see precise notations below), the minimizing $u$ in \eqref{homuellereq} equals the center $c_y$ of the fundamental cell of $L_y$, thus we are justified to fix this choice throughout this section.

\begin{defi}\label{defiLyt}
For any $y\geq 1$ and any $t>0$, we define the body-centred-orthorhombic lattice of lengths $\sqrt{y}$, $1/\sqrt{y}$ and $t$ by
$$
L_{y,t}:=\bigcup_{k\in \Z}(k\tau + \mathbf s(k)+L_y),
$$
where $L_y=\Z(\sqrt{y},0)\oplus\Z(0, \frac{1}{\sqrt{y}})$ is a rectangular lattice of lengths $\sqrt{y}$ and $1/\sqrt{y}$, $\tau=(0,0,t/2)$ and $\mathbf s(k)=c_y=(\sqrt{y}/2,1/(2\sqrt{y}))$ if $k$ is an odd integer and $\mathbf s(k)=0$ if $k$ is an even integer.\\
Define also the lattice
$$
L_{y,t}^*:=\bigcup_{k\in \Z}(k\tau' + \mathbf s'(k)+L_y^*),
$$
where $L_y^*$ is the dual of $L_y$, which is generated by $(1/\sqrt y, \pm \sqrt y)$, $\tau=(0,0,t^{-1})$ and $\mathbf s'(k)=(1/(2\sqrt{y}),\sqrt{y}/2)$ if $k$ is an odd integer and $\mathbf s'(k)=0$ if $k$ is an even integer.\\

Note that $L_{y,t}$ is the anisotropic dilation of the BCC lattice (based on the unit cube) along the coordinate axes by $\sqrt y, 1/\sqrt{y}, t$ respectively, whereas $L_{y,t}^*$ is the dilation of the FCC lattice (again based on the unit cube) by respectively $1/\sqrt y, \sqrt y, 1/t$.
\end{defi}
\begin{lemma} For any $y\geq 1$ and $t>0$, the dual lattice of $L_{y,t}$ is $L_{y,t}^*$.
\end{lemma}
\begin{proof}
Let us consider the generator matrix of $L_{y,t}$ given by
\[
B=\left(
\begin{array}{ccc}
y^{1/2} & 0 & \frac{y^{1/2}}{2} \\
0 & y^{-1/2} & \frac{y^{-1/2}}{2} \\
0 & 0 & t/2
\end{array}\right).
\]
Therefore, the generator matrix of $L_{y,t}^*$ is $B^*=(B^T)^{-1}$, i.e.
\[
B^*=\left(
\begin{array}{ccc}
y^{-1/2} & 0 & 0 \\
0 & y^{1/2} & 0\\
-\frac{1}{t} & -\frac{1}{t} & \frac{2}{t}
\end{array}\right).
\] 
Then we observe that the coordinates of a generic point of $L_{y,t}^*$ are $(c_1,c_2,c_3)=(y^{-1/2}h, y^{1/2}k, 1/t(2n-h-k)), h,k,n\in\mathbb N$, which can be expressed as 
\[
\{(c_1,c_2,c_3):\ \sqrt y c_1+ c_2/\sqrt y + t c_3 \equiv 0 (mod 2)\},
\]
which shows that this lattice corresponds to $L_{y,t}^*$ as in the above definition.
\end{proof}
\noindent For any $\alpha>0$, $t>0$ and $y\geq 1$, the energy of $L_{y,t}$ is given by
\begin{align*}
E_t(y,\alpha):=\theta_{L_{y,t}}(\alpha)&=\theta_3(t^2\alpha)\theta_{L_y}(\alpha)+\theta_2(t^2\alpha)\theta_{L_y+c_y}(\alpha)\\
&=\theta_3(t^2\alpha)\theta_3( y\alpha)\theta_3(\alpha/y)+\theta_2(t^2\alpha)\theta_2(\alpha y)\theta_2(\alpha/y)\\
&=\theta_3(t^2\alpha)\left\{f_{3,\alpha}(y)+\rho_{t,\alpha}f_{2,\alpha}(y)  \right\}
\end{align*}
where, for any $i\in \{2,3\}$, the classical Jacobi theta functions $\theta_i$ are defined by \eqref{jacobithetadef}, the functions $f_{i,\alpha}$ are defined by
$$
\displaystyle f_{i,\alpha}(y):=\theta_i(\alpha y)\theta_i\left( \frac{\alpha}{y} \right)
$$
and
$$
\rho_{t,\alpha}:=\rho_{\Z,1/2}(t^2 \alpha)=\frac{\theta_2(t^2\alpha)}{\theta_3(t^2\alpha)}<1,
$$
by Proposition \ref{ellipticmodulus}. For convenience, we define, for any $(\alpha,t)\in (0,+\infty)^2$ and any $y\geq 1$,
\begin{equation*}
\tilde{E}_t(y;\alpha):=f_{3,\alpha}(y)+\rho_{t,\alpha}f_{2,\alpha}(y).
\end{equation*}
Since $\theta_3(t^2\alpha)>0$ is independent of $y$, the goal of this part is to minimize $y\mapsto \tilde{E}_t(y;\alpha)$ among $y\in[1,+\infty)$, for fixed $\alpha,t$. 
\begin{remark}
The third point of Theorem \ref{thm17} shows that the minimization of $\Lambda\mapsto \theta_\Lambda(\alpha)$ is quite different from the minimization of $\Lambda\mapsto \zeta_\Lambda(s)$ for which FCC and BCC lattices are local minimizers for any $s>0$ (see Ennola \cite{Ennola}). We will present below an algorithm allowing to rigorously check the minimality of $y=1$ for any chosen $\alpha>0$ and any $t>t_0(\alpha)$ (not only for $\alpha=1$).
\end{remark}
\begin{remark}[fixing the unit density constraint]
We note that for $a>0$ the volume of the unit cell of the lattice $aL_y$ is $a^2$ and the distance between layers in the direction orthogonal to $L_y$ after replacing $t$ by $at$ becomes $a|t|/2$. This means that the unit cell of the rescaled layers is of volume $a^3|t|/2$. The constraint of our scaled layered lattice having average density $1$ becomes the requirement that 
\[
a=\frac{2^{1/3}}{t^{1/3}}.
\]
The analogue of $E_t(y,\alpha)$ for the rescaled lattice $a L_{y,t}$ and imposing the unit density requirement above, is 
\[
\theta_{L_{y,t}}(a^2\alpha)=\theta_{L_{y,t}}(\beta), \quad\text{ for }\quad \beta:= \left(\frac{2}{t}\right)^{2/3}\alpha.
\]
If we represent the BCC in our framework we have to use $t=1, y=1$ to start with, and then we have 
\[
\beta_{BCC} = 2^{2/3}\alpha.
\]
For the FCC, after a rotation of $\pi/4$ in the $xy$-plane followed by a dilation (of all three coordinates) of $2^{1/2}$, we see that it corresponds to $y=1, t=\sqrt 2$, and we have
\[
\beta_{FCC} = 2^{1/3}\alpha.
\]
\end{remark}
\noindent We first study the variations of $f_{i,\alpha}$, for given $\alpha>0$ and $i\in \{2,3\}$.
\begin{lemma}
For any $\alpha>0$, we have:
\begin{enumerate}
\item the function $f_{3,\alpha}$ is increasing on $(1,+\infty)$, minimum at $y=1$ and $f_{3,\alpha}''(1)>0$;
\item the function $f_{2,\alpha}$ is decreasing on $(1;+\infty)$, maximum at $y=1$ and $f_{2,\alpha}''(1)<0$.
\end{enumerate}
\end{lemma}
\begin{proof}
The monotonicity of both functions is a direct consequence of \cite[Thm. 2.2]{Faulhuber:2016aa}. Furthermore, we have for any $i\in\{2,3\}$,
$$
f_{i,\alpha}''(1)=\alpha\theta_i''(\alpha)\theta_i(\alpha)-\alpha\theta_i'(\alpha)^2+\theta_i(\alpha)\theta_i'(\alpha),
$$
and our statement about the sign of both second derivatives follows from \cite[Thms. 2.3 and 2.4]{Faulhuber:2016aa}.
\end{proof}
\noindent Then, we study the properties of an associated function $g_\alpha$, given if $\rho_{t,\alpha}=1$:
\begin{prop}\label{cequal1case}
Let $g_\alpha(y):=f_{3,\alpha}(y)+f_{2,\alpha}(y)$, then, for any $\alpha>0$,
\begin{enumerate}
\item $g_\alpha$ is decreasing on $(1,\sqrt{3})$ and increasing on $(\sqrt{3},+\infty)$;
\item we have $g_\alpha ''(1)<0$.
\end{enumerate}
In particular $y=\sqrt{3}$ is the only minimizer of $g_\alpha$ on $[1,+\infty)$.
\end{prop}
\begin{proof}
For any $\alpha>0$, we have, by Poisson summation formula and writing $\beta=\frac{2}{\alpha}$, 
\begin{align*}
g_\alpha(y)&=f_{3,\alpha}(y)+f_{2,\alpha}(y)\\
&=\sum_{m,n\in \Z} e^{-\pi \alpha \left( y m^2 +\frac{n^2}{y} \right)} +\sum_{m,n\in \Z}e^{-\pi \alpha \left( y (m+1/2)^2 +\frac{(n+1/2)^2}{y} \right)} \\
&= \frac{1}{\alpha}\sum_{m,n\in \Z} e^{-\pi \alpha^{-1} \left( y n^2 +\frac{m^2}{y} \right)}+\frac{1}{\alpha}\sum_{m,n\in \Z} e^{-\pi \alpha^{-1} \left( y n^2 +\frac{m^2}{y} \right)}e^{i\pi(m+n)}\\
&=\frac{1}{\alpha}\sum_{m,n\in \Z} e^{-\pi \alpha^{-1} \left( y n^2 +\frac{m^2}{y} \right)}(1+\cos(\pi(m+n))\\
&=\frac{2}{\alpha}\sum_{m,n\in \Z\atop m+n\in 2\Z} e^{-\pi \alpha^{-1} \left( y n^2 +\frac{m^2}{y} \right)}\\
&=\frac{2}{\alpha}\sum_{k,n\in \Z} e^{-\pi\alpha^{-1}\left( y n^2+\frac{(2k-n)^2}{y}  \right)}\\
&=\frac{2}{\alpha}\sum_{k,n\in \Z}  e^{-\frac{\pi}{\alpha}y n^2} e^{-\frac{4\pi}{y\alpha}(k-\frac{n}{2})^2}\\
&=\frac{2}{\alpha}\sum_{k,n\in \Z}  e^{-\pi \beta \frac{y}{2}n^2}e^{-\pi \beta \frac{2}{y}(k-\frac{n}{2})^2}\\
&= \frac{2}{\alpha}\theta\left(\beta;-\frac{1}{2},\frac{y}{2}\right)
\end{align*}
where for the definition of $\theta(\beta; X,Y)$ we use the notations of Montgomery \cite[Eq. (4)]{Mont}. We recall (see  \cite{Mont}) that for any $\beta>0$ the function $(X,Y)\mapsto \theta(\beta;X,Y)$ admits only two critical points: the saddle point $(0,1)$ and the minimizer $\left(\frac{1}{2},\frac{\sqrt{3}}{2}   \right) $, in the half modular domain 
$$
\mathcal{D}:=\left\{(x,y)\in \R^2; 0\leq x\leq 1/2, x^2+y^2\geq 1  \right\}.
$$
Thus we get:
\begin{enumerate}
\item $y=\sqrt{3}$ is the only minimizer of $y\mapsto \theta\left(\beta;-\frac{1}{2},\frac{y}{2}\right)$, because $(-1/2,\sqrt{3}/2)$ and $(1/2,\sqrt{3}/2)$ both correspond to the same (triangular) lattice, by modular invariance $z\mapsto z+1$;
\item $y=1$ is the second critical point, because $(-1/2,1/2)$ and $(0,1)$ both correspond to the same (square) lattice, by modular invariance $z\mapsto  -\frac{1}{z+1}+1$.
\end{enumerate}
Actually, as explained in \cite[Prop. 4.5]{AftBN}, by applying the transformation $z\mapsto -\frac{1}{z+1}+1$, the line $\{(x,y)\in \R^2; x=-1/2, y\geq 1/2\}$ is sent to the quarter of the unit of circle centred at $0$ with extremes $(0,1), (1,0)$, which in particular passes through $\left(\frac{1}{2},\frac{\sqrt{3}}{2}\right)$. For any $\phi \in [0,\pi/2]$, we now have
\begin{align*}
&\tilde{g}(\phi):=\theta(\beta;\cos\phi,\sin\phi)\\
&\tilde{g}'(\phi)=-\sin\phi \partial_x\theta +\cos\phi\partial_y\theta\\
&\tilde{g}''(\phi)=-\cos\phi \partial_x\theta +\sin^2\phi\partial^2_{xx}\theta -\sin\phi \cos\phi \partial^2_{xy}\theta-\sin\phi\partial_y\theta -\cos\phi \sin\phi\partial_{xy}\theta +\cos^2\phi\partial^2_{yy}\theta.
\end{align*}
Since $\partial_y\theta (\beta;0,1)=0$, it follows that $\tilde{g}''(\pi/2)=\partial^2_{xx}\theta(\beta;0,1)$. Let us prove that $\partial^2_{xx}\theta(\beta;0,1)<0$. By Poisson summation formula, we get
\begin{align*}
\theta(\beta;X,Y)&=\sum_{m,n\in \Z} e^{-\pi \beta Y n^2}e^{-\pi\beta(m+nX)^2/Y}\\
&=\sum_{n\in \Z} e^{-\pi \beta Y n^2} \sum_{k\in \Z} e^{-\pi k^2 Y/\beta+2i\pi k nX}\sqrt{\frac{Y}{\beta}}\\
&=\sqrt{\frac{Y}{\beta}}\sum_{n,k\in \Z} e^{-\pi \beta Y n^2} e^{-\pi k^2 Y/\beta}\cos\left(2\pi kn X \right).
\end{align*}
Hence, we obtain
$$
\tilde{g}''(\pi/2)=\partial^2_{XX}\theta(\beta;1,0)=-\frac{4\pi^2}{\sqrt{\beta}}\sum_{n,k} e^{-\pi\beta n^2}e^{-\pi k^2/\beta}k^2n^2<0.
$$
It follows that $g_\alpha''(1)<0$.
\end{proof}
\begin{remark}
In Proposition \ref{cequal1case} we show that the energy $(\Lambda,u)\mapsto \theta_\Lambda(\alpha)+\theta_{\Lambda+u}(\alpha)$ is minimized among rectangular lattices. This energy seems to be minimized by $\Lambda=\frac{1}{3^{1/4}}\Z \times 3^{1/4}\Z$ and $u=\left(\frac{1}{2 \times 3^{1/4}} , \frac{3^{1/4}}{2}\right)$, among all Bravais lattices of density one (see the diagram \cite[Fig 1.e]{Mueller:2002aa} obtained numerically by Mueller and Ho).
\end{remark}
\noindent Thanks to this result, we can prove the first point of Theorem \ref{thm17}.
\begin{lemma}
For any $\alpha>0$ and any $t>0$, $y\mapsto \tilde{E}_t(y;\alpha)$ is an increasing function on $[\sqrt{3},+\infty)$. Consequently, for any $t>0$ and any $\alpha>0$, a minimizer of $y\mapsto \tilde{E}_t(y;\alpha)$ belongs to $[1,\sqrt{3}]$.
\end{lemma}
\begin{proof}
Let $\alpha,t>0$. On $[\sqrt{3},+\infty)$, $g_\alpha'(y)=f_{3,\alpha}'(y)+f_{2,\alpha}'(y)\geq 0$, therefore
$$
\partial_y\tilde{E}_t(y;\alpha)=f_{3,\alpha}'(y)+\rho_{t,\alpha}f_{2,\alpha}'(y)\geq f_{3,\alpha}'(y)(1-\rho_{t,\alpha})> 0
$$
because $f_{3,\alpha}'(y)>0$ and $\rho_{t,\alpha}<1$.
\end{proof}
\noindent We now prove the second point of Theorem \ref{thm17}. We first need the two following lemmas:
\begin{lemma}\label{equih}
For any $y>1$, we have $\tilde{E}_t(y;\alpha)\geq \tilde{E}_t(1;\alpha)$ if and only if
$$
h_\alpha(y):=\frac{f_{3,\alpha}(y)-f_{3,\alpha}(1)}{f_{2,\alpha}(1)-f_{2,\alpha}(y)}\geq \rho_{t,\alpha}.
$$
\end{lemma}
\begin{proof}
For any $y> 1$, $f_{3,\alpha}(y)-f_{3,\alpha}(1)> 0$ and $f_{2,\alpha}(1)-f_{2,\alpha}(y)> 0$, thanks to the minimality of $y=1$ for $f_{3,\alpha}$ and its maximality for $f_{2,\alpha}$.
\end{proof}
\begin{lemma}
For any $\alpha>0$, we have
$$
h_\alpha(1)=\frac{f_{3,\alpha}''(1)}{-f_{2,\alpha}''(1)}=\frac{\alpha\theta_3''(\alpha)\theta_3(\alpha)-\alpha\theta_3'(\alpha)^2+\theta_3(\alpha)\theta_3'(\alpha)}{-\alpha\theta_2''(\alpha)\theta_2(\alpha)+\alpha\theta_2'(\alpha)^2-\theta_2(\alpha)\theta_2'(\alpha)}<1.
$$
\end{lemma}
\begin{proof}
Use a Taylor expansion and the fact that $g_\alpha''(1)=f_{3,\alpha}''(1)+f_{2,\alpha}''(1)<0$.
\end{proof}
\begin{figure}[!h]
\centering
\includegraphics[width=10cm,height=80mm]{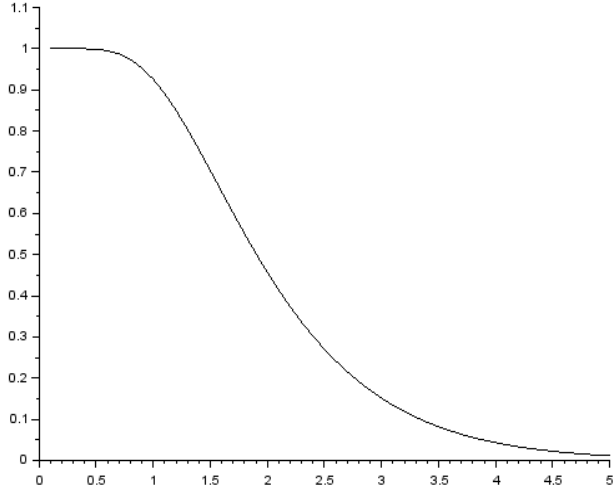}
\caption{Values of $\alpha\mapsto h_\alpha(1)$}
\end{figure}
%
%
%
%
\noindent Therefore, we can prove that $y=1$ is not a minimizer of $\tilde{E}_t(.;\alpha)$ for fixed $\alpha$ and $t$ small enough.
\begin{prop}
For any $\alpha>0$, there exists $t_0(\alpha)$ such that for any $t<t_0(\alpha)$, $h_\alpha(1)<\rho_{t,\alpha}$ and for any $t>t_0(\alpha)$, $h_\alpha(1)>\rho_{t,\alpha}$. In particular, for $t<t_0(\alpha)$, $y=1$ is not a minimizer of $y\mapsto \tilde{E}_t(y;\alpha)$.
\end{prop}
\begin{proof}
By Proposition \ref{ellipticmodulus}, we know that $t\mapsto \rho_{t,\alpha}$ is a decreasing function and $\lim_{t\to 0} \rho_{t,\alpha}=1$. Hence the existence of $t_0(\alpha)$ is proved. Furthermore,  if $t<t_0(\alpha)$, then $h_\alpha(1)<\rho_{t,\alpha}$ and by Lemma \ref{equih}, $y=1$ cannot be a minimizer of $\tilde{E}_t(.;\alpha)$.
\end{proof}
\begin{remark}
In Figure \ref{Fig4}, we have numerically computed the explicit values of $t_0(\alpha)$, $0.1\leq \alpha\leq 10$.
\end{remark}
\begin{figure}[!h]
\centering
\includegraphics[width=10cm,height=80mm]{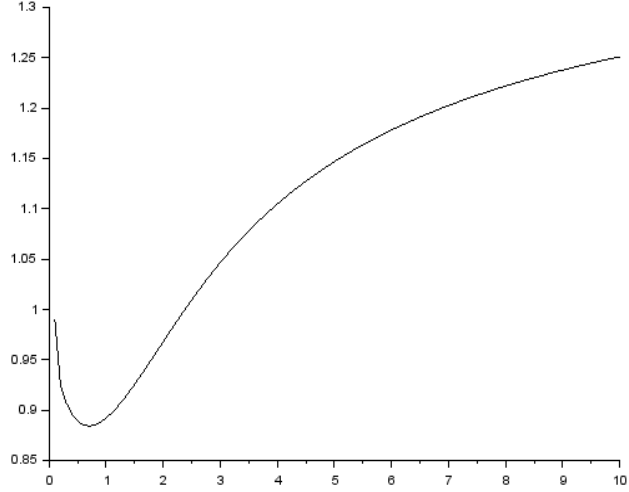}
\caption{Graph of $\alpha \mapsto t_0(\alpha)$, s.t. $h_\alpha(1)=\rho_{t_0(\alpha),\alpha}$, on $[0.1,10]$}
\label{Fig4}
\end{figure}
\noindent Now, fixing $t$ we study the optimality of $y=1$ for large $\alpha$, and we prove the third point of Theorem \ref{thm17} as a corollary.
\begin{prop}\label{Behavinfinity}
For any $t<\sqrt{2}$, there exists $\alpha_t$ such that for any $\alpha>\alpha_t$, $h_\alpha(1)<\rho_{t,\alpha}$, i.e. $y=1$ is not a minimizer of $\tilde{E}_t(.;\alpha)$. Furthermore, for any $t\geq \sqrt{2}$, there exists $\alpha_t$ such that for any $\alpha>\alpha_t$, $h_\alpha(1)>\rho_{t,\alpha}$.
\end{prop}
\begin{proof}
We have
\begin{align*}
h_\alpha(1)-\rho_{t,\alpha}&=\frac{\alpha\theta_3''(\alpha)\theta_3(\alpha)-\alpha\theta_3'(\alpha)^2+\theta_3(\alpha)\theta_3'(\alpha)}{-\alpha\theta_2''(\alpha)\theta_2(\alpha)+\alpha\theta_2'(\alpha)^2-\theta_2(\alpha)\theta_2'(\alpha)}-\frac{\theta_2(t^2\alpha)}{\theta_3(t^2\alpha)}\\
&=\frac{\alpha\theta_3''(\alpha)\theta_3(\alpha)\theta_3(t^2\alpha)-\alpha\theta_3'(\alpha)^2\theta_3(t^2\alpha)+\theta_3(\alpha)\theta_3'(\alpha)\theta_3(t^2\alpha)}{\theta_3(t^2\alpha)\left( -\alpha\theta_2''(\alpha)\theta_2(\alpha)+\alpha\theta_2'(\alpha)^2-\theta_2(\alpha)\theta_2'(\alpha) \right)}\\
&+\frac{\alpha\theta_2''(\alpha)\theta_2(\alpha)\theta_2(t^2\alpha)-\alpha\theta_2'(\alpha)^2\theta_2(t^2\alpha)+\theta_2(\alpha)\theta_2'(\alpha)\theta_2(t^2\alpha)}{\theta_3(t^2\alpha)\left( -\alpha\theta_2''(\alpha)\theta_2(\alpha)+\alpha\theta_2'(\alpha)^2-\theta_2(\alpha)\theta_2'(\alpha) \right)}\\
&\sim \frac{2\pi^2\alpha e^{-\pi\alpha}-2\pi^2\alpha e^{-2\pi \alpha}-\pi e^{-\pi\alpha} -\frac{\pi}{4}e^{-\pi\alpha\left(  (t^2+2)/4 \right)}}{\theta_3(t^2\alpha)\left( -\alpha\theta_2''(\alpha)\theta_2(\alpha)+\alpha\theta_2'(\alpha)^2-\theta_2(\alpha)\theta_2'(\alpha) \right)} \quad \text{as} \quad \alpha\to +\infty.
\end{align*}
Thus, if $\displaystyle \frac{t^2+2}{4}<1$, i.e. $t<\sqrt{2}$, then, as $\alpha\to +\infty$,
$$
h_\alpha(1)-\rho_{t,\alpha}\sim \frac{-\frac{\pi}{4}e^{-\pi\alpha\left(  (t^2+2)/4 \right)}}{\theta_3(t^2\alpha)\left( -\alpha\theta_2''(\alpha)\theta_2(\alpha)+\alpha\theta_2'(\alpha)^2-\theta_2(\alpha)\theta_2'(\alpha) \right)}<0.
$$
Furthermore, if $t\geq \sqrt{2}$, then
$$
h_\alpha(1)-\rho_{t,\alpha}\sim \frac{2\pi^2\alpha e^{-\pi\alpha}}{\theta_3(t^2\alpha)\left( -\alpha\theta_2''(\alpha)\theta_2(\alpha)+\alpha\theta_2'(\alpha)^2-\theta_2(\alpha)\theta_2'(\alpha) \right)}>0
$$
and the proof is completed by Lemma \ref{equih}.
\end{proof}
\begin{remark}
We give two numerical illustrations of this result in Figure \ref{Fig5} (for $t=1$) and Figure \ref{Fig6} (for $t=2$).
\end{remark}
\begin{figure}[!h]
\centering
\includegraphics[width=10cm,height=80mm]{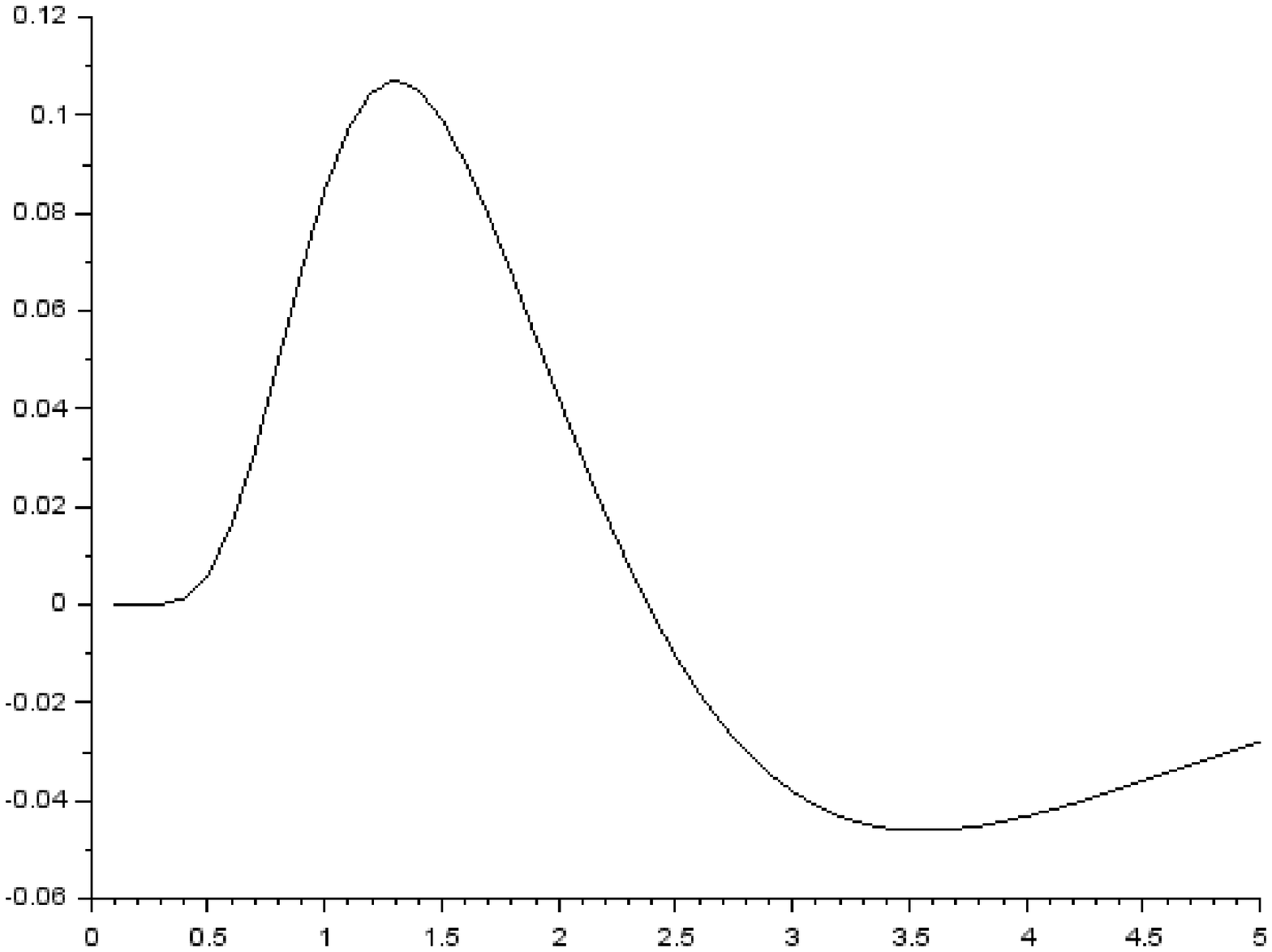}
\caption{Graph of $\alpha\mapsto h_\alpha(1)-\rho_{1,\alpha}$ on $[0.1,5]$}
\label{Fig5}
\end{figure}
\begin{figure}[!h]
\centering
\includegraphics[width=10cm,height=80mm]{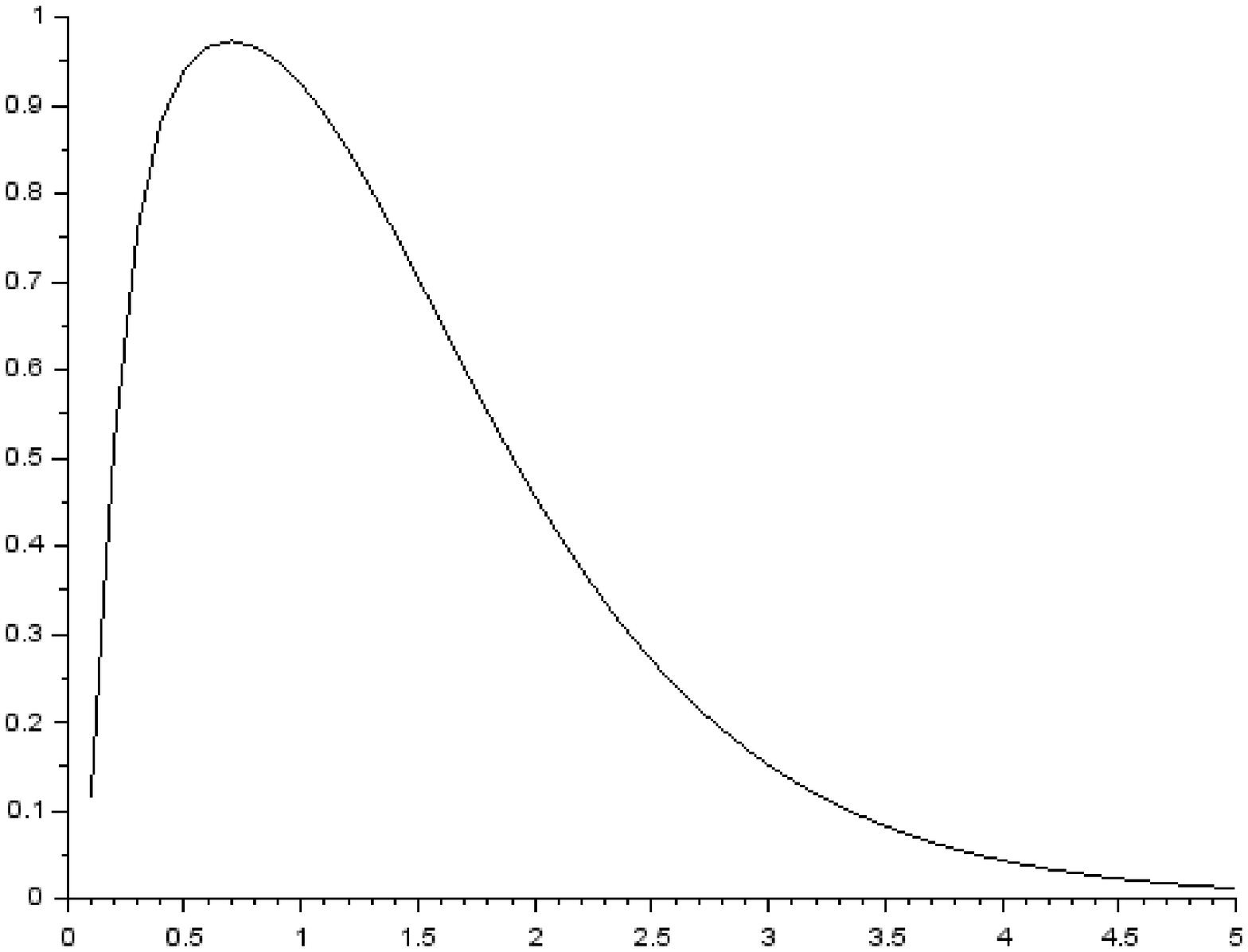}
\caption{Graph of $\alpha\mapsto h_\alpha(1)-\rho_{2,\alpha}$ on $[0.1,5]$}
\label{Fig6}
\end{figure}
\begin{corollary}\textbf{(BCC and FCC are not local minimizers for some $\alpha$)}
There exists $\alpha_1$ such that for any $\alpha>\alpha_1$, $y=1$ is not a minimizer of $y\mapsto \tilde{E}_1(y;\alpha)$, i.e. the BCC lattice is not a local minimizer of $\Lambda\mapsto \theta_\Lambda(\alpha)$ among Bravais lattices of fixed density one.\\
By duality, the FCC lattice is not a local minimizer of $\Lambda\mapsto \theta_\Lambda(\alpha)$ for $\alpha<\alpha_1^{-1}$.
\end{corollary}
\begin{proof}
It is clear by Proposition \ref{Behavinfinity} because $|BCC|=|L_{1,t}|=1\iff t=1<\sqrt{2}$.
\end{proof}
\begin{remark}
As $FCC$ and $BCC$ lattices are local minimizers, at fixed density, of $\Lambda\mapsto \zeta_\Lambda(s)$ for any $s>0$, we notice that the situation for the theta function is more complex and depends on the density of the lattice (i.e. on $\alpha$). Numerically, we find that $\alpha_1\approx 2.38$, and then $\alpha_1^{-1}\approx 0.42$.
\end{remark}
\noindent We finally investigate the global minimum of $\tilde{E}_t(.;\alpha)$ for some values of $\alpha$ and we prove the fourth point of Theorem \ref{thm17}, using a computer-assisted proof.
\begin{lemma}
For any $\alpha>0$, we have $\partial_{yy}^2\tilde{E}_t(1;\alpha)>0$ if and only if $t>t_0(\alpha)$.
\end{lemma}
\begin{proof}
As $\partial_{yy}^2\tilde{E}_t(1;\alpha)=f_{3,\alpha}''(1)+\rho_{t,\alpha}f_{2,\alpha}''(1)$, we get
\begin{align*}
\partial_{yy}^2\tilde{E}_t(1;\alpha)>0 \iff \frac{f_{3,\alpha}''(1)}{-f_{2,\alpha}''(1)}<\rho_{t,\alpha} \iff h_\alpha(1)<\rho_{t,\alpha} \iff t>t_0(\alpha).
\end{align*}
\end{proof}
%
%
\noindent Numerical investigations justify the following conjecture:
\begin{Conjecture}
For any $\alpha>0$ and for any $t>t_0(\alpha)$, $y\mapsto \tilde{E}_t(y;\alpha)$ is a increasing function on $(1,\sqrt{3}]$. And then, $y=1$ is the only minimizer of $\tilde{E}_t(.;\alpha)$.
\end{Conjecture}
\noindent In order to prove the conjecture for some values of $(\alpha,t)$, we need the following result:
\begin{lemma}\label{lemalgo}
Let $y_0<y_2$ and $f\in C^3([y_0,y_2])$ be such that
\begin{enumerate}
\item $f'(y_0)\geq 0$;
\item $f''(y_0)>0$;
\item $\displaystyle \|f'''\|_\infty:= \max_{y\in [y_0,y_2]}|f'''(y)|<K$.
\end{enumerate}
Let $\displaystyle y_1:=\frac{f''(y_0)+\sqrt{f''(y_0)^2+2Kf'(y_0)}}{K}$, then for any $y\in (y_0,y_0+y_1]$, $f'(y)>0$.
\end{lemma}
\begin{proof}
By Taylor expansion, we know that, for any $y\geq y_0$,
$$
f'(y)\geq f'(y_0)+f''(y_0)(y-y_0)-\frac{\|f'''\|_\infty}{2}(y-y_0)^2.
$$
For any $X\geq 0$, let $P(X):=-AX^2+BX+C$, where $A=\frac{\|f'''\|_\infty}{2}$, $B=f''(y_0)$ and $C=f'(y_0)$. It is clear, by definition of $f$, that $P$ admits two roots of opposite signs. The positive one is
$$
X_1:=\frac{f''(y_0)+\sqrt{f''(y_0)^2+2f'(y_0)\|f'''\|_\infty}}{\|f'''\|_\infty},
$$
and $P(X)>0$ on $(0,X_1)$. Now we notice that $x\mapsto \frac{a+\sqrt{a+bx}}{x}$ is a decreasing function of $x$ for any $a,b\geq 0$. Thus, since $\|f'''\|_\infty<K$, we get $X_1>y_1$. Therefore, for any $y\in (y_0,y_0+y_1]$, $f'(y)>0$.
\end{proof}
\begin{lemma}
For any $t>0$ and any $\alpha>0$, 
\begin{align*}
\|\tilde{E}_t'''\|_\infty:=\sup_{y\in [1,\sqrt{3}]}|\tilde{E}_t'''(y;\alpha)|<K_\alpha
\end{align*}
where
\begin{align*}
K_\alpha:&=\sum_{i=2}^3 \alpha^3 |\theta_i'''(\alpha)|\theta_i(\alpha/\sqrt{3})+\alpha^3\theta_i(\alpha)|\theta_i'''(\alpha/\sqrt{3})|+3\alpha^3\theta_i''(\alpha)|\theta_i'(\alpha/\sqrt{3})|\\
&\quad +5\alpha^2|\theta_i'(\alpha)||\theta_i'(\alpha/\sqrt{3})|+3\alpha^2|\theta_i'(\alpha)|\theta_i''(\alpha/\sqrt{3})+6\alpha\theta_i(\alpha)\theta_i(\alpha/\sqrt{3})\\
&\quad +2\alpha^2|\theta_i'(\alpha)|\theta_i(\alpha/\sqrt{3})+2\alpha^2\theta_i(\alpha)|\theta_i'(\alpha/\sqrt{3})|+4\alpha^2\theta_i(\alpha)\theta_i''(\alpha/\sqrt{3}).
\end{align*}
\end{lemma}
\begin{proof}
We get easily the result by using the fact that $|\rho_{t,\alpha}|<1$ and the decreasing of $t\mapsto |\theta^{(k)}(t)|$ on $(0,+\infty)$ for any $k\in \N$.
\end{proof}
\noindent \textbf{Algorithm allowing for checking the conjecture for fixed $\alpha,t$:} We start by defining
\begin{align*}
&y_0=1,\\
&a_0=\tilde{E}_t'(1;\alpha)=0,\\
&b_0=\tilde{E}_t''(1;\alpha)>0.
\end{align*}
Now, by recursion and using both previous lemmas, if $a_i\geq 0$ and $b_i>0$, while $y_i<\sqrt{3}$, we compute
\begin{align*}
&y_{i+1}=y_i+\frac{\tilde{E}_t''(y_i;\alpha)+\sqrt{\tilde{E}_t''(y_i;\alpha)^2+2K_\alpha \tilde{E}_t'(y_i;\alpha)}}{K_\alpha},\\
&a_{i+1}=\tilde{E}_t'(y_{i+1};\alpha),\\
&b_{i+j}=\tilde{E}_t''(y_{i+1};\alpha).
\end{align*}
We have just a finite number of first and second derivatives of $\tilde{E}_t(.;\alpha)$ to compute, which we do numerically. Thus, it is possible to rigorously check that $y\mapsto \tilde{E}_t(y;\alpha)$ is increasing for every desired $\alpha$ and every $t>t_0(\alpha)$.
\begin{lemma}
If $y\mapsto \tilde{E}_{t'}(y;\alpha)$ is increasing on $[1,\sqrt{3})$ for some $\alpha>0$ and $t'>t_0(\alpha)$, then it is always true for any $t\geq t'$.
\end{lemma}
\begin{proof}
Note that $t\mapsto \tilde{E}_t'(y;\alpha)$ is increasing for any $y,\alpha$ since $t\mapsto \rho_{t,\alpha}$ is decreasing and $f_{2,\alpha}'(y)\leq 0$ for any $y\geq 1$.
\end{proof}
\noindent Therefore, the fourth point of Theorem \ref{thm17} is proved via the following result:
\begin{prop}
For $\alpha=1$ and any $t\geq 0.9$, $y=1$ is the unique minimizer of $y\mapsto \tilde{E}_t(y;\alpha)$. In particular, for $t=1$, the BCC lattice $L_{1,1}$ is the only minimizer of $L_{y,1}\mapsto \theta_{L_{y,1}}(1)$.
\end{prop}
\begin{proof}
We use the previous algorithm for $\alpha=1$ and $t=0.9$. We get, after $13$ steps:
\begin{center}
\begin{tabular}{|c|c|c|}
\hline
y & a & b\\
\hline
1 & 0 & 5.7185\\
1.1182 & 0.0009 & 4.2545\\
1.2063 & 0.0024 & 3.4958 \\
1.2792 & 0.0043 & 3.0105\\
1.3428 & 0.0064 & 2.6649 \\
1.4002 & 0.0086 & 2.4017 \\
1.4532 & 0.0108 & 2.1918 \\
1.5029 & 0.0131 & 2.0188 \\
1.5504 & 0.0154 & 1.8724 \\
1.5960 & 0.0177 & 1.7461 \\
1.6403 & 0.0199 & 1.6354 \\
1.6836 & 0.0223 & 1.5372 \\
1.7262 & 0.0246 & 1.4492\\
\hline
\end{tabular}
\end{center}
with $y_{13}=1.7683>\sqrt{3}$. Thus, $y\mapsto \tilde{E}_{0.9}(y;1)$ is increasing on $[1,\sqrt{3})$ and it follows that $y=1$ is the unique minimizer of $y\mapsto \tilde{E}_{t}(y;1)$ for any $t\geq 0.9$.
\end{proof}
\begin{remark} \textbf{Clue for the minimality of BCC and FCC lattices} By Proposition \ref{Behavinfinity} with $|FCC|=1\iff t=\sqrt{2}$, we can reasonably think that there is $\alpha_2$ such that for any $\alpha>\alpha_2$, $y=1$ is the only minimizer of $y\mapsto \tilde{E}_{\sqrt{2}}(y;\alpha)$, i.e. the FCC lattice is the only minimizer of the theta function among lattices $L_{y,t}$ with $t=\sqrt{2}$. Our algorithm allows us to prove this result for $\alpha$ as large as we wish.

\medskip
\noindent  By duality, the BCC lattice is then expected to be a minimizer for any $\alpha<\alpha_2^{-1}$, among dual lattices $L_{y,t}^*$ with $t=1$, of the theta function. 
\end{remark}

\subsection{Geometric families related to FCC and BCC local minimality - Proof of Theorem \ref{thm18}}\label{baernstein}

\noindent As an application of the above two results Corollary \ref{coriwasawa} and Theorem \ref{baernstein}, and of Theorem \ref{thm17}, we find a compelling ``geometric picture'' of possible perturbations of the FCC and BCC lattices, which allows to show their local minimality for some $\alpha>0$ in a geometric way. Indeed, let us focus for the moment on the case of the BCC lattice, and consider the following families of unit density lattices containing it:
\begin{itemize}
\item Let $\mathcal F_1$ be the $2$-dimensional family obtained by considering the representation of Section \ref{bccfccsect} of BCC as layering of triangular lattices. In an adapted coordinate system, the generators of the BCC are the vectors given by rescalings of the generators of $A_2$ and of the translation of the center $c$ of the fundamental triangle of $A_2$. With notations as in Section \ref{bccfccsect} we have:
\begin{eqnarray*}
v_1&:=&\ell_{BCC}(1,0,0)=(2^{5/6},0,0) ,\\
v_2&:=&\ell_{BCC}\left(\frac12,\frac{\sqrt3}{2},0\right)=\left(2^{-1/6},2^{-1/6}3^{1/2},0\right),\\
v_3&:=&\{t_{BCC}\}\times\{\ell_{BCC}\cdot c\}=\left(t_{BCC},\frac{\ell_{BCC}}{2},\frac{\ell_{BCC}}{2\sqrt 3}\right)=\left(2^{-1/6},2^{-1/6}3^{-1/2},2^{-2/3}3^{-1/2}\right).
\end{eqnarray*}
Now the elements of $\mathcal F_1$ will be defined simply by replacing $v_3$ by any vector of the form
\[
v_3':= \{t_{BCC}\}\times\{\ell_{BCC}\cdot x\}, \quad x\in\mathbb R^2.
\]
We obtain for each $x\in\mathbb R^2$ a unit density lattice $\Lambda_x:=\op{Span}_{\mathbb Z}(v_1,v_2,v_3')$ and we can express 
\begin{equation}\label{family1}
\theta_{\Lambda_x}(\alpha)=\sum_{k\in\mathbb Z}e^{-\pi\alpha t_{BCC}^2}\theta_{\ell_{BCC}(A_2+ k x)}(\alpha).
\end{equation}
Due to Theorem \ref{baernstthm}, and using the fact that $A_2+c$ is isometric to $A_2+kc$ for $k\in\mathbb Z$, we find that for all $k\in\mathbb Z, \alpha>0$ there holds 
\[
\theta_{\ell_{BCC}(A_2 + kx)}(\alpha)\le \theta_{\ell_{BCC}(A_2 + c)}(\alpha)= \theta_{\ell_{BCC}(A_2 + kc)}(\alpha)
\]
with equality precisely if $x=c$. Applying this to \eqref{family1} we find that BCC is the unique minimizer within $\mathcal F_1$.
\item Let $\mathcal F_2$ be the $2$-dimensional family obtained by considering the representation of Section \ref{bccfccstack} of BCC as layering of square lattices. In an adapted coordinate system, the generators of the BCC are the vectors, expressed now in terms of the generators of $\mathbb Z^2$ and of the center $c=(1/2,1/2)$ of the unit square of $\mathbb Z^2$, with $\ell=2^{1/3}$ such that our lattice has density $1$:
\begin{eqnarray*}
v_1&:=&\ell(1,0,0)=\left(2^{1/3},0,0\right) ,\\
v_2&:=&\ell\left(0,1,0\right)=\left(0,2^{1/3},0\right),\\
v_3&:=&\{\ell/2\}\times\{\ell\cdot c\}=\left(2^{-2/3},2^{-2/3},2^{-2/3}\right).
\end{eqnarray*}
By the same reasoning as in the previous item, we now define the family $\mathcal F_2$ as formed by the lattices $\Lambda_x'$ with $x\in\mathbb R^2$, obtained by replacing $v_3$ by 
\[
v_3':=\{\ell/2\}\times\{\ell\cdot x\}
\]
and by using now the $d=2$ version of Corollary \ref{coriwasawa} we again obtain that BCC is the unique minimizer of $x'\mapsto\theta_{\Lambda_x'}(\alpha)$ in $\mathcal F_2$.
\item Let $\mathcal F_3^{xy}$ be the family given by the body-centred-orthorhombic lattices where the deformation is within the $xy$-plane like in Section \ref{secorthorombic}. Due to Theorem \ref{thm17} and to the algorithm following Lemma \ref{lemalgo}, we find that BCC is the unique minimizer within $\mathcal F_3^{xy}$ for $\alpha\in\mathcal{A}$ defined by \eqref{defA} and not a minimizer for $\alpha$ large. Indeed, our algorithm is also efficient here to prove this for $\alpha$ as small as we wish. In particular, we rigorously check this fact for any $\alpha\in \mathcal{A}$ defined by
\begin{equation}\label{defA}
\mathcal{A}:=\{0.001k; k\in \N, 1\leq k\leq 100\}.
\end{equation}
This result implies that the FCC lattice is a minimizer for $\alpha\in \mathcal{A}^{-1}$ defined by
\begin{equation}\label{defAmoinsun}
\mathcal{A}^{-1}:=\{1000k^{-1};k\in \N, 1\leq k\leq 1000\}.
\end{equation}
These sets of real numbers are, obviously, only an example and there is no doubt in the fact that these results hold at least for any $\alpha\leq 1$ for the BCC lattice, and for any $\alpha\geq 1$ for the FCC lattice.

\noindent Similar behaviors as for $\mathcal F_3^{xy}$ apply also for the families $\mathcal F_3^{xz}, \mathcal F_3^{yz}$ which are defined precisely like $\mathcal F_3^{xy}$ but with respect to suitably permuted coordinates.

\end{itemize}
Recall that the BCC of density $1$ can be represented as $2^{1/3}\mathbb Z^3 \cup 2^{1/3}(\mathbb Z^2 + (1/2, 1/2, 1/2))$. It is not hard to note that the space of $3$-dimensional lattices $\mathcal L_3^o$ of density $1$ has dimension $5$. Moreover, the tangent spaces $T_{BCC}\mathcal F_1, T_{BCC}\mathcal F_2$ of the families $\mathcal F_1,\mathcal F_2$ span a $3$-dimensional subspace of $T_{BCC}\mathcal L_3^o$, corresponding to the infinitesimal perturbations of the generator $2^{1/3}(1/2,1/2,1/2)$. The tangent spaces of any two of the families $\mathcal F_3^{xy}, \mathcal F_3^{xz}, \mathcal F_3^{yz}$, say $T_{BCC}\mathcal F_3^{xy}$ and $T_{BCC}\mathcal F_3^{xz}$ are easily seen to span the two complementary dimensions. Therefore 
\[
\op{Span}(T_{BCC}\mathcal F_1,T_{BCC}\mathcal F_2,T_{BCC}\mathcal F_3^{xy}, T_{BCC} \mathcal F_3^{xz})=T_{BCC}\mathcal L_3^o.
\]
The above reasoning gives in particular a geometric proof of the local minimality of BCC for $\alpha$ small, and moreover we find a saddle point behavior for $\alpha>\alpha_1$ as in Theorem \ref{thm17}. By applying the FCC version of the results of Section \ref{bccfccsect} (for the family $\mathcal F_1$), Section \ref{bccfccstack} (for the family $\mathcal F_2$), and Theorem \ref{thm17} (which gives the correct reasoning for $\mathcal F_3^{xy}, \mathcal F_3^{xz}, \mathcal F_3^{yz}$) applied now for the case of FCC, we similarly find three families which span $T_{FCC}\mathcal L_3^o$, showing the local minimality of FCC for $\alpha\in \mathcal{A}^{-1}$ defined by \eqref{defAmoinsun}, as well as the saddle point behavior for $\alpha<1/\alpha_1$ from Theorem \ref{thm17}. We thus obtain a geometric proof of Theorem \ref{thm18}.

\medskip

\noindent \textbf{Acknowledgements:} The first author would like to thank the Mathematics Center Heidelberg (MATCH) for support, and Amandine Aftalion for helpful discussions. The second author is supported by an EPDI fellowship. We thank Markus Faulhuber, Sylvia Serfaty, Oded Regev and Xavier Blanc for their feedback on a first version of the paper. We thank the anonymous referee for the useful suggestions and comments, which greatly helped to improve the readability of the paper.
\bibliographystyle{plain}
\bibliography{bibliocrystal}

\vspace{5mm}
\noindent LAURENT B\'ETERMIN\\
Institut f\"ur Angewandte Mathematik and IWR, \\
Universit\"at Heidelberg, \\
Im Neuenheimer Feld 205,\\
 69120 Heidelberg, Germany \\
\texttt{betermin@uni-heidelberg.de}

\medskip
\medskip

\noindent MIRCEA PETRACHE \\
Max Planck Institute for Mathematics in the Sciences,\\
Inselstrasse 22, \\
04103 Leipzig, Germany\\
\texttt{mircea.petrache@mis.mpg.de}

\end{document}